\documentclass[reqno]{amsart}
\usepackage[utf8x]{inputenc}  

\usepackage{setspace}
\usepackage[left=3.1cm, right=3.1cm, bottom=4cm]{geometry}                 

\usepackage{graphicx} 
\usepackage{pgf,tikz}
\usetikzlibrary{arrows}
\usepackage{amssymb}
\usepackage{pdfsync}
\usepackage{mathrsfs}
\usepackage{hyperref} 
\usepackage{verbatim} 
\usepackage{epstopdf}
\usepackage{bbm} 
\usepackage[colorinlistoftodos,prependcaption,textsize=tiny]{todonotes}
\usepackage{xargs}
\usepackage{bm}

\def\R{\mathbb{R}}

\def\N{\mathbb{N}}
\def\Z{\mathbb{Z}}
\def\C{\mathbb{C}}
\def\H{\mathbb{H}}

\def\F{\widehat{F}}

\def\supp{{\rm supp}}
\def\H{\mathcal{H}}

\renewcommand{\d}{\text{\rm d}}

\newcommand{\mc}{\mathcal}

\newtheorem{theorem}{Theorem}

\newtheorem{corollary}[theorem]{Corollary}

\newtheorem*{definition*}{Definition}

\newtheorem{proposition}[theorem]{Proposition}
\newtheorem{lemma}[theorem]{Lemma}

\makeatletter
\DeclareFontFamily{U}{tipa}{}
\DeclareFontShape{U}{tipa}{m}{n}{<->tipa10}{}
\newcommand{\arc@char}{{\usefont{U}{tipa}{m}{n}\symbol{62}}}%

\makeatother

\numberwithin{equation}{section}

\allowdisplaybreaks

\newcommand{\intav}[1]{\mathchoice {\mathop{\vrule width 6pt height 3 pt depth  -2.5pt
\kern -8pt \intop}\nolimits_{\kern -6pt#1}} {\mathop{\vrule width
5pt height 3  pt depth -2.6pt \kern -6pt \intop}\nolimits_{#1}}
{\mathop{\vrule width 5pt height 3 pt depth -2.6pt \kern -6pt
\intop}\nolimits_{#1}} {\mathop{\vrule width 5pt height 3 pt depth
-2.6pt \kern -6pt \intop}\nolimits_{#1}}}

\newcommand{\intavl}[1]{\mathchoice {\mathop{\vrule width 6pt height 3 pt depth  -2.5pt
\kern -8pt \intop}\limits_{\kern -6pt#1}} {\mathop{\vrule width 5pt
height 3  pt depth -2.6pt \kern -6pt \intop}\nolimits_{#1}}
{\mathop{\vrule width 5pt height 3 pt depth -2.6pt \kern -6pt
\intop}\nolimits_{#1}} {\mathop{\vrule width 5pt height 3 pt depth
-2.6pt \kern -6pt \intop}\nolimits_{#1}}}

  \newcommand{\z}{{\bm z}}
    \newcommand{\x}{{\bm x}}
      \newcommand{\y}{{\bm y}}
        \renewcommand{\t}{{\bm t}}

 \newcommand{\ov}{\overline}

\title[Sign uncertainty and de Branges spaces]{Sign uncertainty and de Branges spaces}

\author[Carneiro]{Emanuel Carneiro}
\author[Ismoilov]{Tolibjon Ismoilov}
\author[Ramos]{Antonio Pedro Ramos}

\address{ICTP - The Abdus Salam International Centre for Theoretical Physics, 
Strada Costiera, 11, I - 34151, Trieste, Italy.}
\email{carneiro@ictp.it}

\address{SISSA - Scuola Internazionale Superiore di Studi Avanzati, Via Bonomea 265, 34136 Trieste, Italy}
\email{tolibjon.ismoilov@sissa.it}
\email{adeazeve@sissa.it}

\date{\today}                                           
\begin{document}

\subjclass[2010]{46E22, 42A05, 42B35, 30H45, 33C10}
\keywords{Sign uncertainty, bandlimited functions, de Branges spaces, reproducing kernel.}
\begin{abstract} We investigate here the sign uncertainty phenomenon for bandlimited functions, with a competing condition given by integration with respect to a general measure. Our main result provides a framework related to the theory of de Branges spaces of entire functions, that allows one to find the sharp constants and classify the extremizers in a broad range of situations. We discuss an application in number theory, in connection to bounds for zeros of $L$-functions.
\end{abstract}

\maketitle

\section{Introduction}
\subsection{Background} \label{Background} The expression {\it Fourier uncertainty} is ubiquitous in the literature, generally embodying the paradigm that one cannot have an unrestricted control of a function and its Fourier transform simultaneously. Some uncertainty principles arise motivated by applications in analysis and related fields. A particularly relevant work for our purposes here is the one of Bourgain, Clozel and Kahane \cite{BCK}, in which they introduced the so-called sign Fourier uncertainty principle, and applied it in connection to the study of real zeros of zeta functions over number fields and bounds for the associated discriminants. In order to describe it, let us first recall the basic terminology. 

\smallskip

Throughout this text, our normalization for the Fourier transform of a function $F \in L^1(\R^d)$ is $\F(\y) = \int_{\R^d} e^{-2 \pi i \x \cdot\y} F(\x) \,\d\x$. We say that a function $F:\R^d \to \R$ is {\it eventually non-negative} if $F(\x) \geq 0$ for sufficiently large $|\x|$. In this case, we define the radius of the last sign change of $F$ as 
\begin{equation*}
r(F) :=\inf\{ r>0 \ : \ F(\x) \geq 0 \ \, {\rm for} \ \, |\x|\geq r\}.
\end{equation*}
The sign Fourier uncertainty principle of \cite{BCK} is the quest for the sharp constant
\begin{equation*}
\mathbb{A}(d) := \inf_{0 \neq F\in\mathcal{A}(d)}\sqrt{r(F) \, \cdot\, r\big(\F\,\big)}\,,
\end{equation*}
where the infimum is taken over the class $\mathcal{A}(d)$ of continuous and even functions $F:\R^d \to \R$ such that: (i) $F, \F \in L^1(\R^d)$; (ii) $F(0)\le 0$ and $\F(0)\le 0$; (iii) $F$ and $\F$ are eventually non-negative. Note that the product $r(F)\, \cdot \, r\big(\F\,\big)$ is a natural quantity of interest since it is invariant under rescalings of the function $F$. Bourgain, Clozel and Kahane \cite{BCK} proved that 
\begin{equation*}
\sqrt{\frac{d}{2\pi e}} \leq  \mathbb{A}(d) \leq \sqrt{\frac{d+2}{2\pi}}
\end{equation*}
for any $d \in \N$. In particular, $\mathbb{A}(d) >0$, which is qualitatively the statement that, in this setup, the negative masses of $F$ and $\F$ cannot simultaneously concentrate near the origin. The only known sharp constant in this problem is due to Cohn and Gon\c{c}alves \cite{CG}, who proved that $\mathbb{A}(12) = \sqrt{2}$, using the machinery introduced by Viazovska \cite{Vi} on Laplace transforms of modular forms. Extensions of this problem to a more abstract operator setting can be found in \cite{GOR1}, and to a setup where the signs of $F$ and $\F$ resonate with a prescribed function in \cite{CQE}. Other related works include \cite{GOR2} and \cite{GOS}.

\subsection{A general problem for functions of exponential type} We aim to look at a related problem within the realm of sign uncertainty. In the previous setup, when restricted to the subclass of eventually non-negative {\it bandlimited} $F$, with ${\rm supp}(\widehat{F}) \subset [-\Delta, \Delta]$ (for a $\Delta>0$ prescribed), one could ask for the minimum value of $r(F)$ instead.

\smallskip

The Paley-Wiener theorem \cite[Chapter III, Section 4]{SW} tells us that bandlimited functions are closely related to entire functions of exponential type in $\C^d$, and we can phrase our problem in this more general class. Vectors in $\R^d$ or $\C^d$ are denoted here with bold font (e.g., $\x$, $\y$, $\z$) and numbers in $\R$ or $\C$ with regular font (e.g., $x,y,z$). For $\z = (z_1, z_2, \ldots, z_d) \in \C^d$ we let $|\cdot|$ be the usual norm $|\z| := (|z_1|^2 + \ldots + |z_d|^2)^{1/2}$, and define a second norm $\|\cdot\|$ by 
$$\|\z\| := \sup \left\{ \left|\sum_{n=1}^d z_n\,t_n\right|; \ \t \in \R^d \ {\rm and} \ |\t|\leq 1\right\}.$$
If $F: \C^d \to \C$ is an entire function of $d$ complex variables, which is not identically zero, we say that $F$ has {\it exponential type} if
$$\tau(F):= \limsup_{\|\z\|\to \infty} \|\z\|^{-1}\,\log|F(\z)| < \infty.$$
In this case, the number $\tau(F)$ is called the exponential type of $F$. When $d=1$ this is the classical definition of exponential type and, when $d \geq 2$, this definition is a particular case of a more general concept of exponential type with respect to a compact, convex and symmetric set $K \subset \R^d$ (cf. \cite[pp. 111-112]{SW}). In our case, this convex set $K$ is simply the unit Euclidean ball. With this definition, the Paley-Wiener theorem tells us that for $F \in L^2(\R^d)$ the following are equivalent: (i) $F$ is equal a.e. to the restriction to $\R^d$ of an entire function of exponential type at most $\delta$; and (ii) $\widehat{F}$ is supported in the ball of radius $\delta/2\pi$ centered at the origin.

\smallskip

Let $\mu$ be a locally finite, even and non-negative Borel measure on $\R$ and let $\mu^d$ be its lift to $\R^d$, namely the radial measure given in polar coordinates $\x = r \omega$ by 
\begin{align}\label{20240507_18:19}
\d \mu^d(\x) =  \d \sigma_{d-1}(\omega) \,\d \mu(r)\,,
\end{align}
where $r>0$, $\omega \in \mathbb{S}^{d-1}$, and $\sigma_{d-1}$ is the standard surface measure on the unit sphere $\mathbb{S}^{d-1} \subset \R^d$ (when $d=1$, we simply have $\mu^1 = \mu$). For instance, if $\d\mu(x) = W(x)\,\d x$, then $\d \mu^d(\x)= |\x|^{1-d}\,W(|\x|)\,\d\x$. We say that an entire function $F: \C^d \to \C$ is {\it real entire} when its restriction to $\R^d$ is real-valued. We consider here the following general problem.

\smallskip

\noindent {\it Extremal Problem 1} (EP1): Given a locally finite, even and non-negative Borel measure $\mu$ on $\R$, $\delta \geq 0$ and $d \in \N$, find\footnote{Here the symbol $F\arrowvert_{\R^d}$ denotes $F$ restricted to $\R^d$.}
$$\mathbb{E}(\mu; \delta; d) :=  \inf_{0 \neq F\in\mathcal{E}(\mu; \delta; d)} r(F\arrowvert_{\R^d})\,,$$
where the infimum is taken over the class of functions
\begin{equation*}
\mathcal{E}(\mu; \delta; d) := \left\{
                \begin{array}{ll}
                  F:\C^d \to \C \ \text{real entire of exponential type at most}\  \delta;\\
                  F \in L^1(\R^d,\mu^d) \ \ {\rm and} \ \  \int_{\R^d} F(\x)\, \d \mu^d(\x) \leq 0;\\
                  F\arrowvert_{\R^d}  \ \text{eventually non-negative}.
                \end{array}
              \right\}.
\end{equation*}

\smallskip

With a slight abuse of notation, from now on we identify $F$ and $ F\arrowvert_{\R^d}$, as it will be clear from the context when we are working on $\R^d$. Note the tension between the condition that $F$ is eventually non-negative on $\R^d$ while $\int_{\R^d} F(\x)\, \d \mu^d(\x) \leq 0$. Note also the great generality allowed by the choice of the measure $\mu$. In fact, this problem came to our attention motivated by particular choices of measures $\mu$ arising in applications in number theory; see \S \ref{App1} below. The condition of $\mu$ being locally finite avoids the situation where all functions in the class $\mathcal{E}(\mu; \delta; d)$ need to vanish in a certain region. Note also that the class $\mathcal{E}(\mu; \delta; d)$ might be trivial (i.e., it only contains the identically zero function); this is the case, for instance, when $\d \mu(x) = e^{|x|^4}\d x$, by the classical Hardy's uncertainty principle. Problems related to our (EP1), for bandlimited functions of a constant sign at infinity, have been considered in \cite{GIT, Logan1, Logan2}.

\smallskip

Our first result is that we can reduce matters to dimension $d=1$. 

\begin{theorem} [Dimension shifts] \label{Thm_dimension_shifts} With the notation as above, we have 
\begin{equation*}
\mathbb{E}(\mu; \delta; d)  =\mathbb{E}(\mu; \delta; 1).
\end{equation*}
\end{theorem}

The proof of Theorem \ref{Thm_dimension_shifts} is carried out in Section \ref{Section_2_lifts} and relies on suitable radial symmetrization mechanisms. From now on, when working with $d=1$, let us shorten the notation and simply write:

$$\mathcal{E}(\mu; \delta)  := \mathcal{E}(\mu; \delta; 1)   \ \ \ {\rm and} \ \ \ \mathbb{E}(\mu; \delta) := \mathbb{E}(\mu; \delta; 1).$$

\subsection{Sign uncertainty and de Branges spaces} \label{dB_Subsection} In some situations of interest, we connect the extremal problem (EP1) to the beautiful theory of de Branges spaces of entire functions \cite{Branges}. Let us briefly review here some of the definitions from complex analysis in order to state our results; later in \S \ref{dB_discussion_subs} we provide a more detailed discussion on de Branges spaces.

\smallskip

For an entire function $F :\C \to \C$, we define the entire function $F^*:\C \to \C$ by $F^*(z) := \overline{F(\overline{z})}$. We denote by $\C^+ = \{z \in \C \ ; \ {\rm Im}(z) >0\}$ the open upper half-plane. A {\it Hermite-Biehler} function $E: \C \to \C$ is an entire function that verifies the inequality
\begin{equation}\label{20240506_17:57}
|E^*(z)| < |E(z)|
\end{equation}
for all $z \in \C^+$. We say that an analytic function $F: \C^{+} \to \C$ has {\it bounded type} in $\C^{+}$ if it can be written as the quotient of two functions that are analytic and bounded in $\C^{+}$. If $F: \C^{+} \to \C$ has bounded type in $\C^{+}$, from its Nevanlinna factorization \cite[Theorems 9 and 10]{Branges}, one has 
\begin{equation}\label{20240506_17:58}
v(F) := \limsup_{y \to \infty} \, y^{-1}\log|F(iy)| <\infty.
\end{equation}
The number $v(F)$ is called the {\it mean type} of $F$. If a Hermite-Biehler function $E$ has bounded type in $\C^+$, then so does $E^*$. In fact, if we write $E = P/Q$ with $P$ and $Q$ analytic and bounded in $\C^+$, we have $E^* = ((E^*/E)P) / Q$ and, by a result of Krein (Lemma \ref{theorem-krein} below), $E$ has exponential type and 
$$\tau(E) = \max\{v(E), v(E^*)\} = v(E)\,,$$
where the last equality follows from \eqref{20240506_17:57} and \eqref{20240506_17:58}.

\smallskip

Associated to a Hermite-Biehler function $E: \C \to \C$, one can consider the pair of real entire functions $A$ and $B$ such that $E(z) = A(z) -iB(z)$. These companion functions are given by
\begin{equation}\label{20240508_17:21}
A(z) := \frac12 \big(E(z) + E^*(z)\big) \ \ \ {\rm and}\ \ \  B(z) := \frac{i}{2}\big(E(z) - E^*(z)\big)\,, 
\end{equation}
and note that they can only have real roots by \eqref{20240506_17:57}.

\smallskip

We shall work here with a Hermite-Biehler function $E: \C \to \C$ satisfying the following three conditions:
\begin{itemize}
\item[(C1)] $E$ has bounded type in $\C^+$.
\item[(C2)] The function $z \mapsto E(iz)$ is real entire.
\item[(C3)] $E$ has no real zeros.
\end{itemize}
Note that (C2) is equivalent to the statement that $A$ is even and $B$ is odd, and (C3) is equivalent to that fact that $A$ and $B$ have no common zeros (in particular $A(0) \neq 0$). We consider measures $\mu$ on $\R$ that are related to the measure $|E(x)|^{-2}\d x$ in an integral sense given by the following condition:
\begin{itemize}
\item[(C4)] For any real entire function $F:\C \to \C$ of exponential type at most $2\tau(E)$, that is non-negative on $\R$, we have the identity
\begin{equation}\label{20240712_16:29}
\int_{\R} F(x) \,\d\mu(x) = \int_{\R} F(x) \,|E(x)|^{-2}\,\d x.
\end{equation}
\end{itemize}

\smallskip

Of course, if 
\begin{equation}\label{20240627_16:52}
\d\mu(x) = |E(x)|^{-2}\,\d x\,,
\end{equation}
condition (C4) is trivially verified. As we shall see, it is important for many applications to have the flexibility allowed by the integral condition (C4) rather than the pointwise condition \eqref{20240627_16:52}. 

\smallskip

The next result is the main result of this paper. We fully solve the extremal problem (EP1) in the following general situation.

\begin{theorem}[Sharp constants]\label{dB_Thm2}
Let $E: \C \to \C$ be a Hermite-Biehler function verifying conditions {\rm (C1), (C2)} and {\rm (C3)} above. Let $\mu$ be a locally finite, even and non-negative Borel measure on $\R$ that verifies condition {\rm (C4)}. Let $A := \frac{1}{2} \big(E + E^*\big)$ and let $\xi_1$ be the smallest positive real zero of $A$. Then 
$$\mathbb{E}(\mu; \,2\tau(E)) = \xi_1$$
and the unique extremizer (up to multiplication by a positive constant) is
\begin{equation}\label{20240506_18:44}
F(z) = \frac{A(z)^2}{z^2 - \xi_1^2}.
\end{equation}
\end{theorem}
\noindent {\sc Remarks}: (i) In case $A$ does not have zeros, one may check that $A(z)= a$ and $B(z) = bz$, and hence $E(z) = a -ibz$, for certain $a,b \in \R$ with $ab >0$.  In this case one can show that $\mathcal{E}(\mu; 0) = \{0\}$. 

\smallskip

\noindent (ii) From the proof of Theorem \ref{Thm_dimension_shifts} we also get, for $d >1$, that the unique extremizer (up to multiplication by a positive constant) of $\mathbb{E}(\mu; 2\tau(E); d)$, that is radial on $\R^d$, is the lift of the function \eqref{20240506_18:44}, as discussed in Section \ref{Section_2_lifts}. 

\smallskip

\noindent (iii) Note that the case of the Lebesgue measure and $\delta >0$ is a particular case of Theorem \ref{dB_Thm2} for we can consider $E(z) = e^{- i \delta z/2}$ to show that $\mathbb{E}(\d x; \delta) = \pi/\delta$ and the only extremizer (up to multiplication by a positive constant) is $F(z) = \cos^2(\delta z/2) / (z^2 - (\pi/\delta)^2)$. 

\smallskip

\noindent (iv) For a given measure $\mu$, there might be multiple Hermite-Biehler functions $E$ with the same exponential type that verify (C4); see the discussion in \S\ref{sec5.1}. The smallest positive zero of $A := \frac{1}{2} \big(E + E^*\big)$ is the same for all such $E$.

\smallskip

We shall present two different proofs for Theorem \ref{dB_Thm2} in Section \ref{Sec3_proofs}. The first proof uses a basic polynomial lemma to reduce the search for extremizers to even functions with only one sign change in the positive axis, and in such a setup one can connect matters to a related extremal problem in de Branges spaces, concerning the multiplication by $z$ operator. The second proof makes use of certain quadrature formulas in this context of de Branges spaces.

\subsection{A negative result} In \S \ref{dB_Subsection} we have been interested in a broad situation where we can actually find the sharp version of the uncertainty principle in (EP1). In general, it is an interesting (and seemingly non-trivial) problem to find necessary and sufficient conditions on the measure $\mu$ that would guarantee a qualitative sign uncertainty principle (i.e. the fact that $\mathbb{E}(\mu; \delta) >0$). Our next result shows that, if the measure $\mu$ has a certain amount of mass around the origin and an exponential decay, the sign uncertainty does not hold.

\begin{theorem}\label{Thm3_negative_V2} Let $\mu$ be a locally finite, even and non-negative Borel measure on $\R$ such that:
\begin{itemize}
\item [(i)] $\mu\big([-r, r]\big) > 0$ for each $r >0$. 
\smallskip
\item [(ii)] $\int_{\R} e^{\varepsilon |x|} \,\d \mu(x) <\infty$ for some $\varepsilon >0$.
\end{itemize}
Then $\mathbb{E}(\mu; \delta) = 0$ for any $\delta \geq 0$.  
\end{theorem}

In particular, if the measure $\mu$ verifies (i) and has compact support, we fall under the hypotheses of Theorem \ref{Thm3_negative_V2}. The conditions in Theorem \ref{Thm3_negative_V2} are sufficient but not necessary as the following example shows: if one considers $\mu(x)  = \sum_{n \in \Z} {\bm \delta}(x-n)$, where ${\bm \delta}$ is the usual Dirac delta at the origin, the function $F(z) = \sin^2(\pi z)\in \mathcal{E}(\mu; 2\pi)$ and verifies $r(F) = 0$. If one prefers a function that is not identically zero $\mu$-a.e. one can consider instead $F(z) = (z^2 - \varepsilon^2)\sin^2(\pi z)/(\pi z)^2$ for $\varepsilon >0$ small.

\subsection{Applications}  We discuss here three applications of Theorem \ref{dB_Thm2}.
\subsubsection{Bounds for low-lying zeros in families of $L$-functions} \label{App1} The extremal problem (EP1) is related to the problem of bounding the height of the first zero in families of $L$-functions. Let us very briefly discuss the setup for this problem; the reader can consult the references presented in this subsection for the technical details and a variety of examples.

\smallskip

Let $\mc{F}$ be a set of number theoretical objects (one can think of Dirichlet characters, for example) and assume that to each $f \in \mc{F}$ one can associate an $L$-function
\begin{equation*}
L(s,f)=\sum_{n=1}^\infty \lambda_f(n) \, n^{-s}\,,
\end{equation*}
where the coefficients $\lambda_f(n) \in \C$. Assume that $L(s,f)$ admits an analytic continuation to an entire function. Let us denote the analytic conductor of $L(s,f)$ by $c_f$ and write its non-trivial zeros as $\rho_f=\frac{1}{2}+i\gamma_f$. Assume also the validity of the Generalized Riemann Hypothesis (GRH) for such $L$-functions; i.e., that $\gamma_f \in \R$. The density conjecture of Katz and Sarnak \cite{KS1,KS2} asserts that for each natural family $\{L(s,f), \ f \in \mathcal{F}\}$ of $L$-functions there is an associated symmetry group $G=G(\mathcal{F}$), where $G$ is either: unitary ${\rm U}$, symplectic ${\rm Sp}$, orthogonal ${\rm O}$, even orthogonal ${\rm SO}(\rm{even})$, or odd orthogonal ${\rm SO}({\rm odd})$. Here one wishes to consider averages over $f \in \mc{F}$ ordered by the conductor. Following the notation of Iwaniec, Luo and Sarnak \cite{ILS}, let $\mc{F}(Q):= \{f \in \mc{F} \,: \, c_f = Q\}$\footnote{One may also consider $\mc{F}(Q):= \{f \in \mc{F} \,: \, c_f \leq Q\}$ in some situations.} and let $|\mc{F}(Q)|$ denote its cardinality. If $\phi : \mathbb{R} \to \mathbb{R}$ is an even Schwartz function whose Fourier transform has compact support, Katz and Sarnak \cite{KS1,KS2} conjectured that
\begin{equation}\label{20210430_11:41}
\lim_{Q \to \infty} \, \frac{1}{|\mc{F}(Q)|} \sum_{f\in \mathcal{F}(Q)} \sum_{\gamma_f} \phi\!\left( \gamma_f \frac{\log c_f}{2\pi}\right) = \int_\mathbb{R} \phi(x) \, W_G(x) \, \d x,
\end{equation}
where the sum over $\gamma_f$ counts multiplicity, and $W_G$ is a function (or density) depending on the symmetry group $G(\mathcal{F})$. This is the so-called 1-level density of the low-lying zeros of the family. For these five symmetry groups, Katz and Sarnak determined the density functions:
\begin{equation}\label{20240415_16:35}
\begin{split}
W_{ \rm U}(x)  &=  1\ \ \ ; \ \ \ W_{\rm Sp}(x)  =  1- \frac{\sin 2\pi x}{2\pi x}\ \ \ ; \ \ \ W_{\rm O}(x)  =  1 + \tfrac{1}{2} \boldsymbol{\delta}(x)\,;
\\
W_{{\rm SO}(\rm{even})}(x)  &=  1 + \frac{\sin 2\pi x}{2\pi x}\ \ \ ; \ \ \ 
W_{{\rm SO}({\rm odd})}(x)  =   1- \frac{\sin 2\pi x}{2\pi x}+\boldsymbol{\delta}(x).
\end{split}
\end{equation}

\smallskip

The simplest example would perhaps be when $\mc{F}$ is the set of non-principal Dirichlet characters $\chi$ (and here the analytic conductor of $L(s, \chi)$ is just the modulus $q$ of $\chi$). In this case, Hughes and Rudnick \cite{HR} proved the validity of \eqref{20210430_11:41}, with unitary symmetry $G(\mc{F}) = {\rm U}$, for even Schwartz functions $\phi:\R \to \R$ with $\supp \big(\widehat{\phi}\,\big) \subset [-2,2]$. There are several other results in the literature establishing the validity of \eqref{20210430_11:41} for different families of $L$-functions and different sizes of the support of $\widehat{\phi}$; see, for instance, \cite{CChiM, DPR, FI, HR, ILS, OS} and the references therein.

\smallskip

In this context, the analysis problem that arises is the following: assume the validity of \eqref{20210430_11:41} for even Schwartz functions $\phi:\R \to \R$ with $\supp \big(\widehat{\phi}\,\big) \subset [-\Delta,\Delta]$, with $\Delta >0$ fixed; what is the sharpest upper bound that one can get for the height of the first zero in the family as $Q \to \infty$ with this information? Under these assumptions, if $\phi$ is such that $\phi(x) \geq 0$ for $|x| \geq r$, \,and  $\int_\mathbb{R} \phi(x) \, W_G(x) \, \d x <0 $, relation \eqref{20210430_11:41} plainly yields
$$\limsup_{Q \to \infty} \min_{\substack{ \gamma_f  \\ f \in \mc{F}(Q)}} \left|\frac{\gamma_{f} \,\log c_f}{2\pi}\right| \leq r.$$
Hence, one is led to the problem of finding the infimum over all such possible values of $r$, which is exactly our extremal problem (EP1) in the case of the five measures $\mu$ with densities given by \eqref{20240415_16:35}\footnote{The fact that the functions are even and Schwartz, and that the integral is strictly less than $0$, are harmless due to standard approximation arguments.}. By the Paley-Wiener theorem one seeks here the constant $\mathbb{E}(\mu; 2\pi \Delta)$.

\smallskip

There are three previous works in the literature that address the problem of estimating the height of the first zero within this framework. The first one, by Hughes and Rudnick \cite[Theorem 8.1]{HR}, deals with the case of $G(\mc{F}) = {\rm U}$. The other two, by Bernard \cite[Theorem 1 and Proposition 2]{Bernard} and by the first author with Chirre and Milinovich \cite[Theorem 9]{CChiM}, deal with the other measures. These papers have a common feature: 
they work under the additional assumption that $\phi$ is of the form
\begin{equation}\label{20240507_16:39}
\phi(x) = (x^2 - r^2)\,\psi(x) \,\psi^*(x),
\end{equation}
where $\psi$ is such that $\supp \big(\widehat{\psi}\,\big) \subset [-\Delta/2,\Delta/2]$. Within this more restrictive class, they arrive at the extremal solution. However, in principle, it is not obvious that one can restrict the search in the original problem to functions as in \eqref{20240507_16:39}, which have only one sign change. The cases of these five measures with densities given by \eqref{20240415_16:35} fall under the general umbrella of our Theorem \ref{dB_Thm2}, and this provides a subtle, yet conceptually interesting, contribution to this framework, namely:
\begin{corollary}\label{Cor4_NT}
The additional assumption \eqref{20240507_16:39} brings no loss of generality in the search for $\mathbb{E}(\mu; 2\pi \Delta)$ in the cases of the measures $\mu$ with densities given by \eqref{20240415_16:35}.
\end{corollary}
Hence, the results obtained in \cite{Bernard, CChiM, HR} are in fact the extremal solutions in the larger class $\mc{E}(\mu; 2\pi \Delta)$.

\subsubsection{Power weights} An interesting situation in our framework is the following. Let $\nu >-1$ and consider the locally finite, even and non-negative Borel measure $\mu_{\nu}$ on $\R$ given by 
\begin{equation}\label{20240416_15:54}
\d \mu_{\nu}(x) = |x|^{2\nu +1}\d x.
\end{equation}
Let $J_{\nu}$ be the classical Bessel function of the first kind. Then 
\begin{align}\label{20240710_08:38}
A_{\nu}(z) :=  \Gamma(\nu +1) \left(\tfrac12 z\right)^{-\nu} J_{\nu}(z) = \sum_{n=0}^{\infty} \frac{(-1)^n \big(\tfrac12 z\big)^{2n}}{n!(\nu +1)(\nu +2)\ldots(\nu+n)}
\end{align}
defines an even entire function of exponential type $1$. When $\nu = -1/2$ (which is the case of the Lebesgue measure on $\R$) we simply have $A_{-1/2}(z) = \cos z$. The following result comes as a corollary of our Theorem \ref{dB_Thm2}.

\begin{corollary} \label{Cor_Powers} For $\nu >-1$ let $\mu_{v}$ be given by \eqref{20240416_15:54}. Let ${\frak j}_{\nu, 1}$ be the first positive zero of the Bessel function $J_{\nu}$. Then, for $\delta >0$,
\begin{equation*}
\mathbb{E}(\mu_{\nu}; \delta) = \frac{2\,{\frak j}_{\nu, 1}}{\delta}.
\end{equation*}
Moreover, the unique extremizer for $\mathbb{E}(\mu_{\nu}; \delta)$ (up to multiplication by a positive constant) is 
\begin{equation}\label{20240507_17:08}
F(z) = \frac{A_{\nu}(\delta z/2)^2}{ z^2 - \frac{4{\frak j}_{\nu, 1}^2}{\delta^2}}.
\end{equation} 
\end{corollary}
\noindent {\sc Remarks}: (i) Note that for $\delta =0$ we have $\mathcal{E}(\mu_{\nu}; 0) = \{0\}$.

\smallskip

\noindent (ii) From the proof of Theorem \ref{Thm_dimension_shifts} we also get, for $d >1$, that the unique extremizer (up to multiplication by a positive constant) of $\mathbb{E}(\mu_{\nu}; \delta; d)$, that is radial on $\R^d$, is the lift of the function \eqref{20240507_17:08}, as discussed in Section \ref{Section_2_lifts}. 

\subsubsection{Measures asymptotically comparable to power weights} One may wonder how broad is the reach of Theorem \ref{dB_Thm2}. In fact, fairly general measures would fall under the scope of this theorem (and, in particular, the sign uncertainty principle holds), as our next result illustrates. 

\begin{corollary}\label{Cor6_PWeights}
Let $\mu$ be a locally finite, even and non-negative Borel measure on $\R$,  given by 
$$\d\mu(x) = W(x)\, \d x$$ 
on the set $\{x \in \R \,;\, |x| > R\}$ for some $R >0$, where $W$ is a non-negative Borel measurable function. Assume further that 
$$W(x) \simeq |x|^{2\nu +1}$$
on the set $\{x \in \R \,;\, |x| > R\}$, for some $\nu >-1$. Then, for any $\delta >0$, the problem of computing $\mathbb{E}(\mu; \delta)$ falls under the framework of Theorem \ref{dB_Thm2}. In particular, $\mathbb{E}(\mu; \delta)>0$.
\end{corollary}

\noindent {\sc Remark}: (i) Note that, in the setup of Corollary \ref{Cor6_PWeights}, for $\delta =0$ we have $\mathcal{E}(\mu; 0) = \{0\}$.

\subsection{Notation} We say that $A \lesssim B$ if there is a positive constant $C$ such that $A \leq CB$. We say that $A \simeq B$ when $A \lesssim B$ and $B \lesssim A$. The characteristic function of a set $X$ is denoted by ${\bf 1}_X$.

\section{Lifts and symmetrization: proof of Theorem \ref{Thm_dimension_shifts}}\label{Section_2_lifts}

\subsection{Preliminaries} \label{Lifts_Sec} The radial symmetry of the problem plays an important role in this argument, and we start by recalling two suitable constructions from the work of Holt and Vaaler \cite[Section 6]{HV}.

\smallskip

If $G:\C\to \C$ is an even entire function with power series representation
\begin{align*}
G(z) = \sum_{k=0}^\infty c_k z^{2k}\,,
\end{align*}
we define its {\it lift} to $\C^d$, namely the function ${\mc L}_d(G):\C^d \to \C$, by
\begin{align*}
{\mc L}_d(G)({\z}) := \sum_{k=0}^\infty c_k (z_1^2 + \ldots + z_d^2)^k.
\end{align*}

\smallskip

If $d>1$ and $F: \C^d \to \C$ is an entire function, we define its {\it radial symmetrization} $\widetilde{F}:\C^d \to \C$ by 
\begin{equation*}
\widetilde{F}(\z) := \int_{SO(d)}F(R\z)\,\d \eta (R),
\end{equation*}
where $SO(d)$ denotes the compact topological group of real orthogonal $d \times d$ matrices with determinant $1$, with associated Haar measure $\eta$ normalized so that $\eta(SO(d)) = 1$. If $d=1$, we simply set $\widetilde{F}(z) := \tfrac12\{F(z) + F(-z)\}$. 

\smallskip

The following result is a compilation of \cite[Lemmas 18 and 19]{HV}, and is concerned with the relation between these symmetrization mechanisms and the exponential type of the underlying functions.

\begin{lemma}\label{type-to-ntype} 
The following propositions hold:
\smallskip
\begin{itemize}
\item[(i)] Let $G: \C \to \C$ be an even entire function. Then $G$ has exponential type if and only if ${\mc L}_d(G)$ has exponential type, and $\tau(G) = \tau({\mc L}_d(G))$. 

\smallskip

\item[(ii)] Let $F: \C^d \to \C$ be an entire function. Then $\widetilde{F}:\C^d \to \C$ is an entire function with power series expansion of the form
\begin{equation*}
\widetilde{F}({\z}) = \sum_{k=0}^\infty c_k (z_1^2 + \ldots + z_d^2)^k.
\end{equation*}
Moreover, if $F$ has exponential type then $\widetilde{F}$ has exponential type and $\tau\big(\widetilde{F}\, \big) \leq \tau(F)$.

\end{itemize}
\end{lemma}

\subsection{Proof of Theorem \ref{Thm_dimension_shifts}} Let $F \in \mathcal{E}(\mu; \delta; d)$. Then, from Lemma \ref{type-to-ntype} (ii) we have that $\tau\big(\widetilde{F}\, \big) \leq \tau(F) \leq \delta$. The fact that $F$ is real entire, not identically zero and eventually non-negative implies that $\widetilde{F}$ is also real entire, not identically zero and eventually non-negative, and we plainly have $r\big(\widetilde{F}\, \big) \leq r(F)$. Using Fubini's theorem and the fact that $\mu^d$ is rotationally invariant, we have
\begin{align}\label{20240508_14:51}
\begin{split}
\int_{\R^d} \big| \widetilde{F}(\x)\big| \,\d \mu^d(\x) & \leq  \int_{\R^d} \int_{SO(d)}|F(R\x)|\,\d \eta(R)\,\d \mu^d(\x) \\
& = \int_{SO(d)}  \int_{\R^d}|F(R\x)|\,\d \mu^d(\x)\,\d \eta(R)\\
& =  \int_{SO(d)}  \int_{\R^d}|F(\x)|\,\d \mu^d(\x)\,\d \eta(R)\\
& = \int_{\R^d}|F(\x)|\,\d \mu^d(\x).
\end{split}
\end{align}
Hence $\widetilde{F} \in L^1(\R^d,\mu^d)$. An identical computation to \eqref{20240508_14:51}, without the absolute values, yields
\begin{align*}
\int_{\R^d} \widetilde{F}(\x) \,\d \mu^d(\x) = \int_{\R^d} \int_{SO(d)}F(R\x)\,\d \eta(R)\,\d \mu^d(\x) = \int_{\R^d}F(\x)\,\d \mu^d(\x) \leq 0.
\end{align*}
We then conclude that $\widetilde{F} \in \mathcal{E}(\mu; \delta; d)$ (in sum, we can restrict ourselves to real entire functions that are radial on $\R^d$ without loss of generality).

\smallskip

Using Lemma \ref{type-to-ntype} we see that $\widetilde{F} = {\mc L}_d(G)$ for a certain $G: \C \to \C$ real entire, even, eventually non-negative, and with $\tau(G) = \tau\big(\widetilde{F}\, \big) \leq \delta$. Moreover, from \eqref{20240507_18:19} we have 
\begin{align}\label{20240508_16:56}
\int_{\R^d} \big|\widetilde{F}(\x)\big| \,\d \mu^d(\x) = \int_0^{\infty}\int_{{\mathbb S}^{d-1}} |G(r)|  \, \d \sigma_{d-1}(\omega) \,\d \mu(r) = \frac{\omega_{d-1}}{2} \int_{\R} |G(x)| \,\d\mu(x)\,,
\end{align}
where $\omega_{d-1}$ is the surface area of ${\mathbb S}^{d-1}$, and hence $G \in L^1(\R, \mu)$. Similarly,
\begin{align}\label{20240508_16:57}
0 \geq \int_{\R^d} \widetilde{F}(\x) \,\d \mu^d(\x) = \int_0^{\infty}\int_{{\mathbb S}^{d-1}} G(r)  \, \d \sigma_{d-1}(\omega) \,\d \mu(r) = \frac{\omega_{d-1}}{2} \int_{\R} G(x) \,\d\mu(x)\,,
\end{align}
and we conclude that $G \in  \mathcal{E}(\mu; \delta; 1)$. Since $r(G) = r\big(\widetilde{F}\, \big)$ we arrive at the inequality
\begin{equation}\label{20240508_17:01}
\mathbb{E}(\mu; \delta; 1)  \leq \mathbb{E}(\mu; \delta; d).
\end{equation}

\smallskip

Conversely, starting now with $H \in \mathcal{E}(\mu; \delta; 1)$, since the measure $\mu$ is even, we can replace it by $\widetilde{H}(z) := \tfrac12\{H(z) + H(-z)\}$ that verifies $r\big(\widetilde{H}\, \big) \leq r(H)$. We then consider the lift ${\mc L}_d\big(\widetilde{H}\big)$, and note that Lemma \ref{type-to-ntype} and computations as in \eqref{20240508_16:56} and \eqref{20240508_16:57}
show that ${\mc L}_d\big(\widetilde{H}\big) \in  \mathcal{E}(\mu; \delta; d)$. Since $r\big(\widetilde{H}\big) = r\big({\mc L}_d\big(\widetilde{H}\big)\big)$, we arrive at the inequality
\begin{equation}\label{20240508_17:02}
\mathbb{E}(\mu; \delta; d)  \leq \mathbb{E}(\mu; \delta; 1).
\end{equation}

The result now follows from \eqref{20240508_17:01} and \eqref{20240508_17:02}.

\section{De Branges spaces and the proof of Theorem \ref{dB_Thm2}}\label{Sec3_proofs}

\subsection{Preliminaries} We start by reviewing some of the basic elements from the theory of de Branges spaces, and collecting a few lemmas from the literature that will be important for our purposes. 

\subsubsection{De Branges spaces} \label{dB_discussion_subs} We provide a brief overview, building up on the material introduced in \S \ref{dB_Subsection}. For a detailed exposition we refer the reader to \cite[Chapters I and II]{Branges}.

\smallskip

 Given a Hermite-Biehler function $E: \C \to \C$, the de Branges space $\mc{H}(E)$ associated to $E$ is the space of entire functions $F:\C \to \C$ such that
\begin{equation}\label{20240508_18:22}
\|F\|_{\mc{H}(E)}^2 := \int_{\R} |F(x)|^{2} \, |E(x)|^{-2} \, \d x <\infty\,,
\end{equation}
and such that $F/E$ and $F^*/E$ have bounded type and non-positive mean type in $\C^{+}$. This is a Hilbert space with inner product given by
\begin{equation*}
\langle F, G \rangle_{\mc{H}(E)} :=  \int_{\R}F(x) \, \ov{G(x)} \, |E(x)|^{-2} \, \d x.
\end{equation*}
In fact, this is a reproducing kernel Hilbert space. That is, for each $w \in \C$ the evaluation functional $F \mapsto F(w)$ is bounded on $\mc{H}(E)$ or, equivalently, by the Riesz representation theorem, there exists a function $K(w,\cdot) \in \mc{H}(E)$ such that 
$$F(w) = \langle F, K(w,\cdot) \rangle_{\mc{H}(E)}$$
for all $F \in \mc{H}(E)$. Writing $E(z) = A(z) - iB(z)$ with $A$ and $B$ real entire as in \eqref{20240508_17:21}, the reproducing kernel $K(w,\cdot)$, is given by (see \cite[Theorem 19]{Branges})
\begin{equation}\label{Def_Rep_Ker}
K(w,z) = \frac{B(z)A(\ov{w}) - A(z)B(\ov{w})}{\pi (z - \ov{w})}\,,
\end{equation}
and, when $z = \ov{w}$, 
\begin{equation*}
K(\ov{z}, z) = \frac{B'(z)A(z) - A'(z)B(z)}{\pi }.
\end{equation*}

\smallskip

We shall only be interested in the situation where $K(w,w) = \|K(w, \cdot)\|^2_{\mc{H}(E)} > 0$ for all $w \in \C$, which is equivalent (see \cite[Lemma 11]{HV}) to the statement that $E$ has no real zeros. A phase function $\varphi:\R \to \R$ associated with $E$ is a continuous function such that $E(x)e^{i \varphi(x)}$ is real for all values of $x$. In this case,
$$\varphi'(x) = \frac{\pi \,K(x,x)}{|E(x)|^2} >0.$$ 
If $\alpha$ is a given real number, the solutions of $\varphi(\xi) \equiv\, \alpha\, {\rm (mod \ }\pi{\rm)}$ correspond to the zeros (which must be real and simple) of the real entire function 
\begin{align}\label{20240713_13:17}
T_{\alpha}(z) := \frac{e^{i \alpha}E(z) - e^{-i \alpha}E^*(z)}{2i}.\end{align}
Note that the zeros of any two $T_{\alpha}$ and $T_{\beta}$ interlace. With this notation, observe that $T_0 = -B$ and $T_{\pi/2} = A$. Observe also that there is at most one real number $\alpha$ modulo $\pi$ such that $T_{\alpha} \in \mc{H}(E)$ (otherwise we would have $E \in \mc{H}(E)$, which would contradict \eqref{20240508_18:22}). One of the cornerstones of the theory is the following orthogonality result \cite[Theorem 22]{Branges}: the set of functions $\Gamma_{\alpha}:= \{K(\xi, \cdot)\ ;  T_{\alpha}(\xi) = 0\}$ is always an orthogonal set in $\mc{H}(E)$; if $T_{\alpha} \notin \mc{H}(E)$, the set $\Gamma_{\alpha}$ is an orthogonal basis of $\mc{H}(E)$ and, if $T_{\alpha} \in \mc{H}(E)$, the only elements of $\mc{H}(E)$ that are orthogonal to $\Gamma_{\alpha}$ are the constant multiples of $T_{\alpha}$. In particular, if $T_{\alpha} \notin \mc{H}(E)$, for every $F \in \mc{H}(E)$ we have
\begin{equation}\label{20210809_11:01}
F(z) = \sum_{T_{\alpha}(\xi) = 0}  \frac{F(\xi)}{K(\xi, \xi)} \, K(\xi, z) \ \ \ {\rm and} \ \ \   \|F\|_{\mc{H}(E)}^2 = \sum_{T_{\alpha}(\xi) = 0}  \frac{\big| F(\xi)\big|^2 }{K(\xi, \xi)}.
\end{equation}

\subsubsection{Krein's memoir} We recall a result of M. G. Krein \cite{K} that allows us to relate the condition of bounded type in $\C^+$ to the condition of exponential type in $\C$. Recall that $\log^+|x| :=\max\{0, \log |x|\}$.

\begin{lemma}[Krein]\label{theorem-krein} 
Let $F:\C \to \C$ be an entire function. The following conditions are equivalent:
\begin{enumerate}
\item[(i)] $F$ and $F^*$ have bounded type in the open upper half-plane $\C^+$.

\smallskip

\item[(ii)] $F$ has exponential type and 
\begin{equation*}
\int_{\R} \frac{\log^+|F(x)|}{1+x^2} \, \d x<\infty.
\end{equation*}
\end{enumerate}
If either and therefore both of these conditions hold then $\tau(F) = \max\{v(F), v(F^*)\}$.
\end{lemma}

Let $\mathcal{B}$ denote the set of entire functions $F$ which satisfy one and therefore both of the conditions (i) or (ii) in Lemma \ref{theorem-krein}. The next result can be proved using Lemma \ref{theorem-krein} and is contained in \cite[Lemma 12]{HV}. It allows us to easily characterize the de Branges space $\mc{H}(E)$ when $E$ has bounded type in $\C^+$.

\begin{lemma} \label{Lem_9_charac_dB_spaces_ET} \label{Lem_9_VET} Let $E$ be a Hermite-Biehler function of bounded type in $\C^+$, and $F$ be an entire function. Then $E$ belongs to $\mathcal{B}$ and the following are equivalent:
\begin{itemize}
\item[(i)] $F \in \mc{H}(E)$.
\item[(ii)] $F \in \mathcal{B}$, $\max\{v(F), v(F^*)\} \leq v(E)$ and \eqref{20240508_18:22} holds.
\item[(ii)] $F$ has exponential type, $\tau(F) \leq \tau(E)$ and \eqref{20240508_18:22} holds.
\end{itemize}
Moreover, if $\xi$ is real and $0 < K(\xi, \xi)$, the entire function $K(\xi, \cdot) \in \mc{H}(E)$ verifies $\tau(K(\xi, \cdot)) = \tau(E)$. 
\end{lemma}

Our final preliminary lemma in this subsection is \cite[Lemma 14]{CL1}, a factorization result for real entire functions of exponential type that are non-negative on $\R$. This is an important tool to connect the $L^1$ and the $L^2$ theories in our context.

\begin{lemma}\label{nonnegl1-tohe} 
Let $E$ be a Hermite-Biehler function of bounded type in $\C^+$ with exponential  type $\tau(E)$. Let $F:\C \to \C$ be a real entire function of exponential type at most $2\tau(E)$ that satisfies
\begin{align*}
F(x)\ge 0
\end{align*}
for all $x \in \R$ and
\begin{align*}
\int_{\R} F(x)\, |E(x)|^{-2} \,\d x <\infty.
\end{align*} 
Then there exists $U\in\H(E)$ such that
\begin{align*}
F(z) = U(z) \,U^*(z)
\end{align*}
for all $z \in \C$.
\end{lemma}

\subsubsection{Polynomial reductions} \label{SubSec_pol_red} We consider in this subsection a basic auxiliary result that allows us to reduce the search for extremizers in our extremal problem (EP1) to functions that have only one sign change. The next proposition is related to \cite[Problem 45, Part VI]{PS} and we provide here a short proof for the reader's convenience.
\begin{proposition}\label{Lem_polynomial}
    Let $n\ge 1$ be an integer and let $0<t_1<t_2<\ldots<t_n$ be distinct positive real numbers. If $P$ is the monic polynomial defined by 
    \begin{equation*}
        P(x)=\prod_{k=1}^{n}(x^2-t_k^2),
    \end{equation*}
then, for any positive number $r>t_n$, there exist even polynomials $Q$ and $R$, both non-negative on $\R$, satisfying the following properties:
\smallskip
\begin{itemize}
\item[(i)] $P(x)=(x^2-r^2)Q(x) +R(x)$ for all $x\in \mathbb{R}$.

\medskip

\item[(ii)] ${\rm deg}(Q)= 
\begin{cases}
 2n-2, & \textit{ if } n \textit{ is odd;} \\
 2n-4, & \textit{ if } n \textit{ is even;} 
 \end{cases}$
  \ \ \ and \ \ \
  ${\rm deg}(R)= \begin{cases}
   2n-2, & \textit{ if } n \textit{ is odd;} \\
                2n, &  \textit{ if } n \textit{ is even.} 
\end{cases}$

\end{itemize}
\end{proposition}

\begin{proof} We prove this by induction on $n$. At each step we show how to construct the polynomials $Q = Q_n$ and $R = R_n$ (it should be clear to the reader that these polynomials depend not only on $n$ but on $t_1, t_2, \ldots, t_n$ and $r$ as well). If $n=1$ and $0<t_1<r$, note that 
    \begin{equation*}
        P(x)=x^2-t_1^2=(x^2-r^2)\cdot 1+(r^2-t_1^2).
    \end{equation*}
    In this case we can take $Q_1(x):=1$ and $R_1(x):=r^2-t_1^2$, which are polynomials of degree 0.
    
 \smallskip
 
Now let $n=2$ and $0<t_1<t_2<r$. Since $P(x)=(x^2-t_1^2)(x^2-t_2^2)$, we can apply the previous case to each of the factors as follows
    \begin{equation*}
    \begin{split}
        P(x)&=\big((x^2-r^2) \,Q_1(x)+R_1(x)\big)\left((x^2-r^2)\cdot 1+(r^2-t_2^2)\right)\\
        &=(x^2-r^2)\left\{R_1(x)+(r^2-t_2^2)\,Q_1(x)\right\}+\left\{(x^2-r^2)^2 \,Q_1(x)+(r^2-t_2^2)\,R_1(x)\right\}.\\
    \end{split}
    \end{equation*}
In this case we can take $Q_2(x):=R_1(x)+(r^2-t_2^2)\,Q_1(x)$ and $R_2(x):=(x^2-r^2)^2 \,Q_1(x)+(r^2-t_2^2)\,R_1(x)$. Note that since ${\rm deg}(Q_1)={\rm deg}(R_1)=0$ we have ${\rm deg}(Q_2)=0$ and ${\rm deg}(R_2)=4$.   

\smallskip
    
    Assume that the assertion holds for all $n\le N-1$, for some integer $N\ge 3$. Consider the case $n=N$, with $0<t_1<t_2<\ldots<t_{N-1}<t_{N}<r$ and write 
    \begin{equation}\label{20240509_14:36}
        P(x)=(x^2-t_{N}^2)\prod_{k=1}^{N-1}(x^2-t_k^2).
    \end{equation}
    By the induction hypothesis, there exist even polynomials $Q_{N-1}(x)$ and $R_{N-1}(x)$, both non-negative on $\R$, such that
    \begin{equation}\label{20240509_14:37}
        \prod_{k=1}^{N-1}(x^2-t_k^2)=(x^2-r^2)\,Q_{N-1}(x) +R_{N-1}(x).
    \end{equation}
    On the other hand, the factor $(x^2-t_N^2)$ falls into the scope of the case $n=1$; i.e., 
    \begin{equation}\label{20240509_14:38}
    x^2-t_N^2=(x^2-r^2)\cdot 1+(r^2-t_N^2).
    \end{equation} 
    Combining \eqref{20240509_14:36}, \eqref{20240509_14:37} and \eqref{20240509_14:38} we get
    \begin{equation*}
    \begin{aligned}
        P(x)&= \big((x^2-r^2)\,Q_{N-1}(x) +R_{N-1}(x)\big)\,\big((x^2-r^2)\cdot 1+(r^2-t_N^2)\big)\\
        &=(x^2-r^2)\left\{R_{N-1}(x)+(r^2-t_N^2)\,Q_{N-1}(x)\right\}+\left\{(x^2-r^2)^2 \,Q_{N-1}(x)+(r^2-t_N^2)\,R_{N-1}(x)\right\}\\
        &=(x^2-r^2)\,Q_{N}(x) +R_{N}(x),
    \end{aligned}
    \end{equation*} 
    where $Q_{N}(x):=R_{N-1}(x)+(r^2-t_N^2)\,Q_{N-1}(x)$ and $R_{N}(x):=(x^2-r^2)^2 \,Q_{N-1}(x)+(r^2-t_N^2)\,R_{N-1}(x)$. Note that $Q_N$ and $R_N$ are both even and non-negative on $\R$. Also, by the induction hypothesis,
$${\rm deg}(Q_{N})=\max\big\{{\rm deg}(Q_{N-1}), {\rm deg}(R_{N-1})\big\}={\rm deg}(R_{N-1})=
\begin{cases}
 2N-4, & \textnormal{ if } N \textnormal{ is even;} \\
                2N-2, & \textnormal{ if } N \textnormal{ is odd;}
\end{cases}$$ 
and
$${\rm deg}(R_{N})=\max\big\{4+{\rm deg}(Q_{N-1}), {\rm deg}(R_{N-1})\big\}=4+{\rm deg}(Q_{N-1})=
\begin{cases}
 2N, & \textnormal{ if } N \textnormal{ is even;} \\
                2N-2, & \textnormal{ if } N \textnormal{ is odd.}
\end{cases}$$
This concludes the proof.
    \end{proof}
    
With the aid of Proposition \ref{Lem_polynomial} we establish the following reduction in the search for extremizers in problem (EP1). 
 
 \begin{proposition}\label{Prop11_restriction}
 Let $F \in \mathcal{E}(\mu; \delta)$ be an even function and set $r := r(F)$. Then there exists $G \in \mathcal{E}(\mu; \delta)$ of the form $G(z) = (z^2 - r^2) H(z)$, where $H$ is an even and real entire function of exponential type at most $\delta$ that is non-negative on $\R$. In particular, note that $r(G) = r(F)$. 
 \end{proposition}
\begin{proof}
Since $F$ is even and belongs to $\mathcal{E}(\mu; \delta)$, it has only finitely many sign changes on the positive axis. If it has only one sign change, we are done. If not, we label these as $0 < t_1 < t_2 < \ldots < t_n < r$ and write
\begin{equation}\label{20240509_18:22}
F(x) = \left(\prod_{k=1}^n (x^2 - t_k^2) \right) (x^2 - r^2) \,F_1(x)\,,
\end{equation}
where $F_1$ is an even and real entire function of exponential type at most $\delta$ that is non-negative on $\R$. We apply Proposition \ref{Lem_polynomial} to the product from $k=1$ to $n$ in \eqref{20240509_18:22} and observe that 
\begin{align}\label{20240509_20:26}
\begin{split}
F(x) & = \Big((x^2-r^2)Q(x) +R(x)\Big)(x^2 - r^2) \,F_1(x)\\
& = (x^2-r^2)^2 \,Q(x) F_1(x) + (x^2 - r^2) R(x)F_1(x)\\
& \geq (x^2 - r^2) R(x)F_1(x)
\end{split}
\end{align}
for all $x \in \R$. Let $G(z) := (z^2 - r^2) R(z)F_1(z)$. By Proposition \ref{Lem_polynomial} we have ${\rm deg}(R) \leq 2n$ and hence
\begin{equation}\label{20240509_20:25}
|R(x)| \lesssim  \left|\prod_{k=1}^n (x^2 - t_k^2) \right| 
\end{equation}
for $|x|$ large. Since $F \in L^1(\R, \mu)$, from \eqref{20240509_20:25} and the fact that $\mu$ is locally finite we conclude that $G \in L^1(\R, \mu)$. Moreover, from \eqref{20240509_20:26} we plainly have
\begin{equation*}
0 \geq\int_{\R} F(x) \,\d \mu(x) \geq \int_{\R} G(x) \,\d \mu(x).
\end{equation*}
Hence, $G \in  \mathcal{E}(\mu; \delta)$, as desired.
\end{proof}

\subsection{Proof of Theorem \ref{dB_Thm2}} We are now in position to present our first proof of Theorem \ref{dB_Thm2}. Assume now that we are considering a measure $\mu$ and a Hermite-Biehler function $E$ that verify the conditions (C1), (C2), (C3) and (C4).

\subsubsection{Reduction to an $L^2$-extremal problem} \label{Sub-Sec_reduction_L2} As already observed in the proof of Theorem \ref{Thm_dimension_shifts}, if we start with $F \in \mathcal{E}(\mu; 2\tau(E))$, we can consider $\widetilde{F}(z) := \tfrac12\{F(z) + F(-z)\}$ that verifies $\widetilde{F} \in \mathcal{E}(\mu; 2\tau(E))$ (since $\mu$ is even) and $r\big(\widetilde{F}\, \big) \leq r(F)$. Hence, for the purposes of finding the sharp constant, we can restrict our search to even functions. We have seen in Proposition \ref{Prop11_restriction} that we can further restrict our search to functions $G \in \mathcal{E}(\mu; 2\tau(E))$ of the form $ G(z) = (z^2 - r^2) H(z)$, with $H$ even, real entire of exponential type at most $2 \tau(E)$, and non-negative on $\R$. Since $x^2H(x) \in L^1(\R, \mu)$ (recall that $\mu$ is locally finite), condition (C4) yields
\begin{align}
\begin{split}\label{20240510_11:54}
0 \geq \int_{\R} G(x) \,\d\mu(x) & = \left( \int_{\R} x^2H(x) \,\d\mu(x)\right) -  r^2\left( \int_{\R} H(x)\, \d\mu(x)\right)\\
& = \left( \int_{\R} x^2H(x)\, |E(x)|^{-2} \,\d x\right) -  r^2\left( \int_{\R} H(x)\, |E(x)|^{-2} \,\d x\right)\\
& = \int_{\R} G(x) \,|E(x)|^{-2} \,\d x.
\end{split}
\end{align}
Since $H \in L^1(\R, |E(x)|^{-2}\,\d x)$ (from condition (C4)), we may use Lemma \ref{nonnegl1-tohe} to decompose $H = UU^*$, with $U \in \mathcal{H}(E)$. Hence, as in \eqref{20240510_11:54}, the fact that $0 \geq \int_{\R} G(x) \,\d\mu(x)$ is equivalent to the fact that 
$$r^2 \geq \frac{ \int_{\R} x^2\,|U(x)|^2\, |E(x)|^{-2} \,\d x} { \int_{\R}|U(x)|^2\, |E(x)|^{-2} \,\d x} = \frac{\|z \, U\|_{\mc{H}(E)}^2}{\|U\|_{\mc{H}(E)}^2}.$$

\smallskip

As our task is to minimize $r$, we have reduced our extremal problem (EP1) to the following extremal problem related to multiplication by $z$ in a general de Branges space.

\smallskip

\noindent {\it Extremal Problem 2} (EP2): Let $E$ be a Hermite-Biehler function. Find
\begin{equation*}
\mathbb{M}_z(E) := \inf_{0 \neq U \in \mc{H}(E)} \frac{\|z \, U\|_{\mc{H}(E)}}{\|U\|_{\mc{H}(E)}}.
\end{equation*}
It turns out that this problem has recently been addressed in the work of the first author with Chirre and Milinovich \cite[Theorem 14]{CChiM}, for $E$ verifying conditions (C2) and (C3) (note that condition (C1) was already used above when we invoked Lemma \ref{nonnegl1-tohe}). We discuss its solution in the next subsection.

\subsubsection{Solution of the extremal problem {\rm (EP2)}} \label{SS_3.2.2}We recall here the result in \cite[Theorem 14]{CChiM} and briefly reproduce its proof for the reader's convenience.  An extension of this result, computing the analogous infimum related to multiplication by $z^k$, $k \in \N$, was obtained in \cite{CGOOORS}.
\begin{proposition}[cf. \cite{CChiM}]\label{Prop_CCM}
Let $E$ be a Hermite-Biehler function verifying conditions {\rm (C2)} and {\rm (C3)}. Let $A := \frac{1}{2} \big(E + E^*\big)$ and let $\xi_1$ be the smallest positive real zero of $A$. Then 
\begin{equation*}
\mathbb{M}_z(E) = \xi_1.
\end{equation*}
The unique extremizer (up multiplication by a non-zero complex constant) is
\begin{equation}\label{20240707_13:15}
U(z) = \frac{A(z)}{z^2 - \xi_1^2}.
\end{equation}
\end{proposition}

\noindent {\sc Remark:} If $A$ has no zeros then, from the discussion in \S \ref{dB_discussion_subs}, the space $\mc{H}(E)$ has dimension at most $1$ and there are no functions $U \in \mc{H}(E)$ such that $zU \in \mc{H}(E)$.

\begin{proof} Let $\mc{X}(E) := \{U \in \mc{H}(E)\, : \,  zU \in \mc{H}(E)\}$, and recall that $B = \frac{i}{2} \big(E - E^*\big)$. The proof considers two cases: when $A \notin \mc{H}(E)$ (a typical situation), and when $A \in \mc{H}(E)$ (an exceptional situation).

\smallskip 

We first consider the case $A \notin \mc{H}(E)$. In this case, if $0 \neq U \in \mc{X}(E)$, from \eqref{20210809_11:01} we have 
\begin{align}\label{20210909_23:48}
\|U\|_{\mc{H}(E)}^2  = \sum_{A(\xi) = 0}  \frac{\big| U(\xi)\big|^2 }{K(\xi, \xi)} \leq \frac{1}{\xi_1^2} \sum_{A(\xi) = 0}  \frac{|\xi|^2\,\big|U(\xi)\big|^2 }{K(\xi, \xi)} = \frac{1}{\xi_1^2}\, \|z U\|_{\mc{H}(E)}^2. 
\end{align}
This shows that $\mathbb{M}_z(E)  \geq \xi_1$. In order to have equality in \eqref{20210909_23:48} one must have $U(\xi) = 0$ for each zero $\xi$ of $A$ with $\xi \neq \pm \xi_1$. By the interpolation formula in \eqref{20210809_11:01} (recall that $K(\xi_1, \xi_1) = K(-\xi_1, -\xi_1)$ in our setup) we get
\begin{align}\label{20210913_17:12}
U(z) & = \frac{1}{K(\xi_1, \xi_1)} \big( U(\xi_1)\, K(\xi_1, z) +  U(-\xi_1) \, K(-\xi_1, z)\big).
\end{align}
From \eqref{20210913_17:12} and \eqref{Def_Rep_Ker} one can write 
\begin{equation}\label{20210913_17:19}
zU(z) = \frac{A(z) B(\xi_1)}{\pi K(\xi_1, \xi_1)} \big( U(-\xi_1) -U(\xi_1)\big) + \frac{\xi_1 \, U(\xi_1)}{K(\xi_1, \xi_1)} K(\xi_1, z) + \frac{(-\xi_1) \, U(-\xi_1)}{K(\xi_1, \xi_1)} K(-\xi_1, z).
\end{equation}
From \eqref{20210913_17:19}, one now observes that $zU \in \mc{H}(E)$ if and only if $U(\xi_1) = U(-\xi_1)$. This leads to the conclusion that $\mathbb{M}_z(E) =  \xi_1$, with extremizers as in \eqref{20240707_13:15}.

\smallskip

We now consider the exceptional case $A \in \mc{H}(E)$. By \cite[Theorem 29]{Branges} a function $G \in \mc{H}(E)$ is  orthogonal to the subspace $\mc{X}(E)$ if and only if it is of the form $G(z) = \alpha A(z) + \beta B(z)$ for constants $\alpha,\beta \in \C$. If there is such a function $G$ with $\beta \neq 0$, then $B \in  \mc{H}(E)$, and hence $E \in \mc{H}(E)$, contradicting \eqref{20240508_18:22}. Therefore $\mc{X}(E)^{\perp} = {\rm span}\{A\}$. Recalling the discussion at the end of \S \ref{dB_discussion_subs}, that $T_{\pi/2} = A$ and $\Gamma_{\pi/2} \cup \{A\}$ is an orthogonal basis of $\mc{H}(E)$, if $0 \neq U \in \mc{X}(E)$ we have
\begin{align}\label{20210913_16:40}
\|U\|_{\mc{H}(E)}^2  = \sum_{A(\xi) = 0}  \frac{\big| U(\xi)\big|^2 }{K(\xi, \xi)} \leq \frac{1}{\xi_1^2} \sum_{A(\xi) = 0}  \frac{|\xi|^2\,\big|U(\xi)\big|^2 }{K(\xi, \xi)} \leq \frac{1}{\xi_1^2}\, \|z U\|_{\mc{H}(E)}^2\,,
\end{align}
which shows that $\mathbb{M}_z(E)  \geq \xi_1$. Equality in \eqref{20210913_16:40} occurs if and only if $U(\xi) = 0$ for each zero $\xi$ of $A$ with $\xi \neq \pm \xi_1$, and $zU \perp A$. As we are assuming that $U \in \mc{X}(E)$, we have that $U \perp A$, which implies that $U$ has a representation as in \eqref{20210913_17:12}. From \eqref{20210913_17:19} one sees that $zU \perp A$ if and only if $U(\xi_1) = U(-\xi_1)$, which leads us to the conclusion that $\mathbb{M}_z(E) =  \xi_1$, with extremizers as in \eqref{20240707_13:15}.
\end{proof}

In light of the discussion in \S \ref{Sub-Sec_reduction_L2} and Proposition \ref{Prop_CCM}, we arrive at the conclusion that 
\begin{equation*}
\mathbb{E}(\mu; \,2\tau(E)) = \mathbb{M}_z(E) = \xi_1,
\end{equation*}
as desired.

\subsubsection{Classification of the extremizers} The discussion in \S \ref{Sub-Sec_reduction_L2} and \S \ref{SS_3.2.2} also leads us to the following conclusion: if $G \in \mathcal{E}(\mu; 2\tau(E))$ is an even extremizer that has only one sign change on the positive real axis at $\xi_1$, then $G$ must be of the form \eqref{20240506_18:44}.

\smallskip

Observe that any even function $F \in \mathcal{E}(\mu; 2\tau(E))$ verifies 
\begin{equation}\label{20240707_15:39}
\int_{\R} F(x) \,\d\mu(x) = \int_{\R} F(x) \,|E(x)|^{-2}\,\d x.
\end{equation}
This follows from condition (C4), by writing $F$ as a difference of two functions that are non-negative on $\R$ and belong to $L^1(\R, \mu)$ (looking at the factorization \eqref{20240509_18:22}, this plainly follows by writing an even polynomial as a difference of two polynomials that are non-negative on $\R$). Now assume that $F \in \mathcal{E}(\mu; 2\tau(E))$ is an {\it even extremizer that has more than one sign change on the positive real axis} (i.e. $n \geq 1$ in the factorization \eqref{20240509_18:22}). One then notices that the inequality $F(x) \geq G(x)$ in \eqref{20240509_20:26} is strict for a.e. $x \in \R$. In light of \eqref{20240707_15:39}, we are led to the fact that 
\begin{align}\label{20240707_15:48}
0 \geq \int_{\R} F(x) \,\d\mu(x) = \int_{\R} F(x) \,|E(x)|^{-2}\,\d x > \int_{\R} G(x) \,|E(x)|^{-2}\,\d x = \int_{\R} G(x) \,\d\mu(x).
\end{align}
Recall that $G$ has the form $G(z) = (z^2 - \xi_1^2) H(z)$, with $H$ non-negative on $\R$. In the situation \eqref{20240707_15:48}, the function $G_{\varepsilon}(z) = (z^2 - (\xi_1 - \varepsilon)^2) H(z)$ would belong to $\mathcal{E}(\mu; 2\tau(E))$ for small $\varepsilon$, contradicting the minimality of $\xi_1$. We conclude that the unique even extremizer is indeed the one given by  \eqref{20240506_18:44}.

\smallskip

Finally, we need to show that all extremizers must be even. If $F \in \mathcal{E}(\mu; 2\tau(E))$ is a generic extremizer, we have seen that $\widetilde{F}(z) := \tfrac12\{F(z) + F(-z)\}$ is an even extremizer, and hence given by \eqref{20240506_18:44}. So, we want to show that if 
\begin{equation}\label{20240707_16:04}
\frac12\big\{F(z) + F(-z)\big\} =  \frac{A(z)^2}{z^2 - \xi_1^2}\,,
\end{equation}
then $F(z)=  \frac{A(z)^2}{z^2 - \xi_1^2}$. Since $F$ is an extremizer, $F(x) \geq 0$ for $|x| \geq \xi_1$. Then, by \eqref{20240707_16:04}, for each zero $\xi$ of $A$ we must have $F(\xi) = 0$ (with even multiplicity if $|\xi|>\xi_1$). This tells us that $F$ factors as 
\begin{equation}\label{20240707_16:33}
F(z) = \frac{A(z)^2}{z^2 - \xi_1^2} \, R(z)\,,
\end{equation}
with $R$ real entire and non-negative on $\{x \in \R \ ; \ |x| \geq \xi_1\}$. Then \eqref{20240707_16:04} and \eqref{20240707_16:33} imply that
$$\frac12\big\{R(z) + R(-z)\big\} = 1,$$
and we find that $R$ is bounded on $\R$. Recall that 
\begin{align*}
K(\xi_1, z) = -\frac{A(z)B(\xi_1)}{\pi (z - \xi_1)}  \ \ \ {\rm and}\  \ \ K(-\xi_1, z) = \frac{A(z)B(\xi_1)}{\pi (z + \xi_1)},
\end{align*}
and by  Lemma \ref{Lem_9_charac_dB_spaces_ET} we have $\tau(K(\pm \xi_1, \cdot) ) = v(K(\pm \xi_1, \cdot) )  = \tau(E)$. Hence
\begin{align}\label{20240713_12:44}
v\left(\frac{A(z)^2}{z^2 - \xi_1^2}\right) = v\left(\frac{A(z)}{z - \xi_1}\right) + v\left(\frac{A(z)}{z + \xi_1}\right) = 2\tau(E).
\end{align}
Since $F$ is real entire, by Lemma \ref{Lem_9_charac_dB_spaces_ET}, $v(F) = \tau(F) \leq 2\tau(E)$. Since $R \in \mathcal{B}$ as in Lemma \ref{theorem-krein}, from \eqref{20240707_16:33}, \eqref{20240713_12:44}, and the additive property of the mean type, we arrive at
$$2\tau(E) \geq v(F) = v\left(\frac{A(z)^2}{z^2 - \xi_1^2}\right) + v(R) = 2\tau(E) + v(R)\,,$$
which leads us to the conclusion that $\tau(R) = v(R) = 0$. Now, an application of the Phr\'{a}gmen-Lindel\"{o}f principle \cite[Theorem 1]{Branges} tells us that $R$ has to be identically $1$, as desired. 

\subsection{Alternative proof of Theorem \ref{dB_Thm2}} This second proof makes use of certain quadrature formulas in our context of de Branges spaces. It is slightly more direct than our first proof (although we are going to invoke some elements from the proof of Proposition \ref{Prop_CCM}). On the other hand, the polynomial reduction used in the first proof has the advantage that it can also be used for general measures $\mu$. Here we again consider a measure $\mu$ and a Hermite-Biehler function $E$ that verify the conditions (C1), (C2), (C3) and (C4). We also assume that $A$ has at least one zero.

\subsubsection{Quadrature formulas} First observe that there exists a real entire function $M$ of exponential type at most $2\tau(E)$ that is strictly positive on $\R$ and verifies $M \in L^1(\R, |E(x)|^{-2}\d x)$ (which, by (C4), is equivalent to $M \in L^1(\R, \mu)$). Take, for instance,
\begin{equation*}
M(z) = K(\xi_1, z)^2 +K(0, z)^2 = \frac{ B(\xi_1)^2A(z)^2}{\pi^2 (z - \xi_1)^2} + \frac{ A(0)^2B(z)^2}{\pi^2 z^2}.
\end{equation*} 
Now, for each $F \in \mathcal{E}(\mu; 2\tau(E))$, there is a positive constant $c = c_F$ such that $F +c M$ is non-negative on $\R$. Then  $F +c M \in L^1(\R, \mu)$ and, by (C4), $F +c M \in L^1(\R, |E(x)|^{-2}\d x)$. In particular, $F \in L^1(\R, |E(x)|^{-2}\d x)$. Using (C4) again, we have
\begin{align}\label{20240707_18:26}
\begin{split}
\int_{\R} F(x) \,\d\mu(x) &= \int_{\R} \big(F(x) + cM(x)\big) \,\d\mu(x) - \int_{\R}  cM(x) \,\d\mu(x)\\
& = \int_{\R} \big(F(x) + cM(x)\big) \,|E(x)|^{-2}\,\d x - \int_{\R} cM(x) \,|E(x)|^{-2}\,\d x\\
& = \int_{\R} F(x) \,|E(x)|^{-2}\,\d x.
\end{split}
\end{align}
Using Lemma \ref{nonnegl1-tohe}, we can write $F +c M = UU^*$ and $cM = VV^*$ with $U, V \in \mc{H}(E)$, thus arriving at the representation
\begin{align}\label{20240707_18:11}
F = UU^* - VV^*.
\end{align}

\smallskip

For $\alpha \in \R$, let $T_{\alpha}$ be as in \eqref{20240713_13:17}. If $T_{\alpha} \notin \mc{H}(E)$, from \eqref{20240707_18:11} and \eqref{20210809_11:01} we get
\begin{align}\label{20240707_18:25}
\begin{split}
 \int_{\R} F(x) \,|E(x)|^{-2}\,\d x &=  \int_{\R} |U(x)|^2 \,|E(x)|^{-2}\,\d x - \int_{\R} |V(x)|^2 \,|E(x)|^{-2}\,\d x\\
 & = \sum_{T_{\alpha}(\xi) = 0} \frac{|U(\xi)|^2}{K(\xi, \xi)} -  \sum_{T_{\alpha}(\xi) = 0} \frac{|V(\xi)|^2}{K(\xi, \xi)}\\
 & = \sum_{T_{\alpha}(\xi) = 0} \frac{F(\xi)}{K(\xi, \xi)}.
 \end{split}
\end{align}
Hence, from \eqref{20240707_18:26} and \eqref{20240707_18:25}, we arrive at the quadrature formula
\begin{align}\label{20240707_18:39}
\int_{\R} F(x) \,\d\mu(x) = \sum_{T_{\alpha}(\xi) = 0} \frac{F(\xi)}{K(\xi, \xi)}.
\end{align}

\subsubsection{Finding the sharp constant} \label{20240807_18:07} Assume that there exists $F \in \mathcal{E}(\mu; 2\tau(E))$ such that $r(F) < \xi_1$.  Then there is an interval $(a,b) \subset (r(F), \xi_1)$ in which $F$ is strictly positive. As discussed in \S \ref{dB_discussion_subs}, choose $\alpha$ such that $T_{\alpha} \notin \mc{H}(E)$ and $T_{\alpha}$ has a zero $\zeta \in (a,b)$. Note that such a $T_{\alpha}$ will have no other zero in the interval $(-\xi_1, \xi_1)$ by the interlacing property of the zeros. This plainly contradicts \eqref{20240707_18:39}, since the left-hand side is non-positive by hypothesis, while the right-hand side is strictly positive. Hence, $\mathbb{E}(\mu; \,2\tau(E)) \geq \xi_1$.

\smallskip

In order to see that the function $F$ given by \eqref{20240506_18:44} verifies $\int_{\R} F(x) \,\d\mu(x) = 0$, we can argue as in the proof of Proposition \ref{Prop_CCM}. In the generic case when $A = T_{\pi/2} \notin \mc{H}(E)$ we can, alternatively, use \eqref{20240707_18:39} directly. This shows that  $\mathbb{E}(\mu; \,2\tau(E)) = \xi_1$.

\subsubsection{Classification of the extremizers} Let $F \in \mathcal{E}(\mu; 2\tau(E))$ be an extremizer for our problem. Arguing as in \S \ref{20240807_18:07}, we conclude that $F$ cannot take strictly positive values in the interval $(-\xi_1, \xi_1)$. In particular, since $F$ is continuous, we must have $F(\pm \xi_1) = 0$. This means that $F$ can be factored as $F(z)=(z^2 - \xi_1^2)H(z)$, with $H$ real entire, of exponential type at most $2\tau(E)$, and non-negative on $\R$. From condition (C4), we may use Lemma \ref{nonnegl1-tohe} to decompose $H = UU^*$, with $U \in \mc{H}(E)$. Then, arguing as in the proof of Proposition \ref{Prop_CCM} we arrive at the conclusion that $F$ must be given by \eqref{20240506_18:44}.

\section{Absence of sign uncertainty: proof of Theorem \ref{Thm3_negative_V2}}

\subsection{Density of polynomials} We start with the following auxiliary result.
\begin{lemma} \label{Den_Pol}Let $\mu$ be a locally finite, even and non-negative Borel measure on $\R$ such that 
\begin{equation}\label{20240907_10:07}
\int_{\R} e^{\varepsilon |x|} \,\d \mu(x) <\infty
\end{equation}
for some $\varepsilon >0$. Then the space of polynomials is dense in $L^2\big(\mathbb{R}, (1+x^2) \,\d\mu(x)\big)$.
\end{lemma}

\begin{proof}
Let $ f\in L^2\big(\mathbb{R}, (1+x^2) \,\d\mu(x)\big)$ be such that
\begin{equation}\label{20240709_10:49}
\int_{\R} f(x) \, x^k (1+x^2) \,\d\mu(x) = 0
\end{equation}
for each $k=0,1,2,\ldots$. From \eqref{20240907_10:07} and an application of the Cauchy-Schwarz inequality we have
\begin{equation*}
\begin{split}
    \int_{\R} |f(x) (1 + x^2)|\, \d\mu(x) \leq \left( \int_{\R} (1 + x^2)\, \d\mu(x) \right)^{\frac{1}{2}} \left( \int_{\R} |f(x)|^2 (1 + x^2)\, \d\mu(x) \right)^{\frac{1}{2}} < \infty.
\end{split}
\end{equation*}
Hence $f(x) (1 + x^2) \in L^1(\mathbb{R}, \mu)$. For $\varepsilon >0$ as in \eqref{20240907_10:07}, let $\Omega_{\varepsilon}$ be the region in the complex plane given by $\Omega_{\varepsilon}:=\left\{ z \in \mathbb{C} : | {\rm Im}(z) | < \frac{\varepsilon}{8\pi} \right\}$, and define
\begin{equation*}
F(z) := \int_{\R} e^{-2 \pi i z x} f(x) (1 + x^2) \, \d\mu(x), \quad {\rm for} \ \ z \in \Omega_{\varepsilon}.
\end{equation*}
Note that for each $ z \in \Omega_{\varepsilon}$ we have 
\begin{equation}\label{20240709_10:46}
\begin{split}
    | F(z) | & \leq  \int_{\R} \left\lvert  e^{ -2 \pi i z  x}\right\rvert |f(x)| \,(1 + x^2) \,\d\mu(x) \leq \int_{\R} e^{ \frac{\varepsilon}{4} \lvert x\rvert } |f(x)| \,(1 + x^2) \,\d\mu(x) \\
    & \lesssim \int_{\R} e^{ \frac{\varepsilon}{2} \lvert x\rvert } |f(x)| \sqrt{1 + x^2} \, \d\mu(x) \\
    &  \leq \left( \int_{\R} e^{\varepsilon \lvert x\rvert} \, \d\mu(x) \right)^{\frac{1}{2}} \left( \int_{\R} |f(x)|^2 (1 + x^2)\, \d\mu(x) \right)^{\frac{1}{2}}< \infty.
\end{split}
\end{equation}
Hence, $F$ defines an analytic function in $\Omega_{\varepsilon}$ by Morera's theorem. 

\smallskip

Using \eqref{20240709_10:46} (for the absolute convergence), Fubini's theorem, and \eqref{20240709_10:49}, we have, for each $z \in  \Omega_{\varepsilon}$ with $|z| < \frac{\varepsilon}{8\pi}$,
\begin{align*}
F(z) =   \int_{\R} \left(\sum_{k=0}^{\infty} \frac{(- 2 \pi i z x)^k}{k!}\right) f(x) (1 + x^2) \, \d\mu(x)  = \sum_{k=0}^{\infty}  \frac{(- 2 \pi i z)^k}{k!} \int_{\R}  f(x) \, x^k (1+x^2) \,\d\mu(x)  = 0.
\end{align*}
Since $F$ is analytic on $\Omega_{\varepsilon}$, we conclude that $F =0$ in this region. Since $F$ is the Fourier transform of the finite Borel measure $f(x) (1 + x^2) \, \d\mu(x)$, we conclude that this measure must be identically zero. This in turn implies that $f=0$ in $L^2\big(\mathbb{R}, (1+x^2) \,\d\mu(x)\big)$, as desired.
\end{proof}

\subsection{Proof of Theorem \ref{Thm3_negative_V2}} Let $\mu_1$ be the measure given by $\d\mu_1(x)=(1+x^2)\,\d\mu(x)$. Fix $0 < r <2$ and let $\eta = \eta(r) >0$ to be chosen later. Using Lemma \ref{Den_Pol}, let $P$ be a polynomial such that 
\begin{align}\label{20240709_13:17}
\big\|P - {\bf 1}_{[-r,r]}\big\|_{L^2(\R,\,  \mu_1)} \leq \eta.
\end{align}
Note that 
\begin{align}\label{20240709_13:23}
\begin{split}
\int_{\R} (x^2 - r^2)\, |P(x)|^2\, \d\mu(x)& \leq \int_{[-r,r]^c} (x^2 +1)\, |P(x)|^2\, \d\mu(x) + \int_{-r/2}^{r/2} (x^2 - r^2)\, |P(x)|^2\, \d\mu(x)\\
& \leq \big\|P - {\bf 1}_{[-r,r]}\big\|^2_{L^2(\R,\,  \mu_1)} + \left( \left(\frac{r}{2}\right)^2 - r^2\right)\int_{-r/2}^{r/2} |P(x)|^2\, \d\mu(x)\\
& \leq \eta^2  - \left(\frac{3r^2}{4}\right)\, \frac{1}{2}  \int_{-r/2}^{r/2} (1 + x^2)|P(x)|^2\, \d\mu(x).
\end{split}
\end{align}
From the triangle inequality and \eqref{20240709_13:17} we get
\begin{align*}
\big| \|P\|_{L^2([-\frac{r}{2}, \frac{r}{2}],\,  \mu_1)} - \|{\bf 1}\|_{L^2([-\frac{r}{2}, \frac{r}{2}],\,  \mu_1)} \big| \leq \|P - {\bf 1}\|_{L^2([-\frac{r}{2}, \frac{r}{2}],\,  \mu_1)} \leq \big\|P - {\bf 1}_{[-r,r]}\big\|_{L^2(\R,\,  \mu_1)} \leq \eta,
\end{align*}
and hence
\begin{align}\label{20240709_13:24}
\|P\|_{L^2([-\frac{r}{2}, \frac{r}{2}],\,  \mu_1)} \geq \|{\bf 1}\|_{L^2([-\frac{r}{2}, \frac{r}{2}],\,  \mu_1)} - \eta = \sqrt{\mu_1\big([-\tfrac{r}{2}, \tfrac{r}{2}]\big)} - \eta.
\end{align}
From \eqref{20240709_13:23} and \eqref{20240709_13:24} we obtain
\begin{align}\label{20240709_13:29}
\int_{\R} (x^2 - r^2)\, |P(x)|^2\, \d\mu(x) \leq  \eta^2 - \frac{3r^2}{8}\left( \sqrt{\mu_1\big([-\tfrac{r}{2}, \tfrac{r}{2}]\big)} - \eta\right)^2.
\end{align}
Choosing $\eta$ sufficiently small, we can make the right-hand side of \eqref{20240709_13:29} be negative. Then the function $F(z) = (z^2 - r^2)P(z)P^*(z)$ belongs to the class $\mathcal{E}(\mu; \delta)$ and $r(F) = r$. Since $r$ was arbitrary, we conclude that $\mathbb{E}(\mu; \delta) = 0$.

\section{Applications}

\subsection{Preliminaries} \label{sec5.1} The following alternative characterization of de Branges spaces is given by \cite[Theorem 23]{Branges}. 
\begin{lemma}[cf. \cite{Branges}]\label{Lem15_AC}
Let $H$ be a Hilbert space of entire functions that contains a non-zero element and verifies the following hypotheses: 
\begin{enumerate}
\item[(i)] If $F \in H$ and has a non-real zero $w$, the function $z \mapsto F(z)(z-\overline{w})/(z-w)$ also belongs to $H$ and has the same norm of $F$.
\item[(ii)] For every non-real number $w$, the linear functional defined on $H$ by $F \mapsto F(w)$ is continuous.
\item[(iii)] If $F \in H$, the function $F^*$ also belongs to $H$ and has the same norm of $F$.
\end{enumerate}
Then $H$ is equal isometrically to some de Branges space $\mc{H}(E)$. 
\end{lemma}

We have seen in \S \ref{dB_discussion_subs} how to construct the $K$ from $E$. An observation that is important for applications is the fact that we can revert this process and construct $E$ from $K$ (not necessarily in a unique way). This works as follows: if $K$ is the reproducing kernel of a Hilbert space $H$ verifying the conditions of Lemma \ref{Lem15_AC}, define the function 
\begin{equation}\label{20240709_15:43}
L(w,z) := 2\pi i (\overline{w} - z) K(w,z)\,,
\end{equation}
and the entire function
\begin{equation}\label{20210909_19:00}
E(z) := \frac{L(i,z)}{L(i,i)^{\frac12}}.
\end{equation}
One can show that $E$ is a Hermite-Biehler function such that 
\begin{equation*}
L(w,z) = E(z)E^*(\ov{w}) - E^*(z)E(\ov{w})\,,
\end{equation*}
and that $H$ is equal isometrically to the de Branges space $\mc{H}(E)$. See \cite[Appendix A]{CCLM} for details.

\subsection{Proof of Corollary \ref{Cor4_NT}} For $\Delta >0$ and $W$ one of the density functions given by \eqref{20240415_16:35}, one can consider the Hilbert space $H$ of entire functions $F$ of exponential type at most $\Delta$ such that 
\begin{align*}
\|F\|_{H}:= \left(\int_{\R} |F(x)|^2\, W(x)\,\d x \right)^{1/2} < \infty.
\end{align*}
The usual Paley-Wiener space $PW(\Delta)$ is the (reproducing kernel) Hilbert space of entire functions $F$ of exponential type at most $\Delta$ such that 
\begin{align*}
\|F\|_{PW(\Delta)}:= \|F\|_{L^2(\R)} =  \left(\int_{\R} |F(x)|^2\,\d x \right)^{1/2} < \infty.
\end{align*}
One sees that $H$ and $PW(\Delta)$ are equal as sets (since the density function $W$ is asymptotically $1$). Moreover, due to classical uncertainty principles for the Fourier transform, one can show that 
$$\|F\|_{H} \simeq  \|F\|_{L^2(\R)}.$$
See \cite[Section 3]{CChiM} for details. We then conclude that $H$ is also a reproducing kernel Hilbert space, and fits into the hypotheses of Lemma \ref{Lem15_AC}.

\smallskip

It is shown in \cite[\S 4.1]{CChiM} that the reproducing kernel $K$ of $H$ verifies the following properties, for all $w,z \in \C$: (i) $K(w,w) >0$; \, (ii) $K(-w,-z) = K(w,z)$; \, (iii) $K(z,w) = \overline{K(w,z)}$; \, (iv) $\overline{K(\overline{w},\overline{z})} = K(w,z)$. It then follows that $E$ defined by \eqref{20240709_15:43} - \eqref{20210909_19:00} verifies conditions (C2) and (C3) (if $E(w) = 0$ for $w \in \R$, one would have $K(w,w) = 0$). Since $K(i,z)$ has exponential type, so does $E$. Also, since $K(i,z)$ belongs to $L^2(\R)$, we may use Jensen's inequality to obtain
\begin{align}\label{20240712_16:00}
\frac{1}{\pi} \int_{\R} \frac{\log^+|E(x)|}{1+x^2} \, \d x \leq \frac{1}{2\pi} \int_{\R} \frac{\log \big(1 + |E(x)|\big)^2}{1+x^2} \, \d x \leq \frac{1}{2} \log\left( 1 + \frac{1}{\pi} \int_{\R} \frac{|E(x)|^2}{1 + x^2} \,\d x \right) < \infty.
\end{align}
Hence, $E$ verifies the conditions of Lemma \ref{theorem-krein}, which implies it has bounded type in $\C^+$ (condition (C1)). From Lemma \ref{Lem_9_VET} we have $\tau(E) = \tau(K(0,\cdot))$, which is at most $\Delta$. In fact, we must have $\tau(E) = \Delta$, since by Lemma \ref{Lem_9_VET} the space $H = \mc{H}(E)$ admits only functions $F$ with $\tau(F) \leq \tau(E)$, and we know from the start that there exist functions in $PW(\Delta)$ with exponential type equal to $\Delta$. Finally, condition (C4) follows from the classical Krein's factorization in $PW(\Delta)$ or Lemma \ref{nonnegl1-tohe}, depending if we start with the left-hand side or the right-hand side of \eqref{20240712_16:29} being finite, and the fact that $H$ is isometric to $\mc{H}(E)$.

\smallskip

The conclusion is that this situation falls in the framework of Theorem \ref{dB_Thm2}, as claimed. 

\subsection{Proof of Corollary \ref{Cor_Powers}} \label{Pf_Cor_5} Let $\nu > -1$. A de Branges space $\mc{H}(E)$ is said to be homogeneous of order $\nu$ if, for all $0 < a < 1$ and all $f \in \mc{H}(E)$, the function $z \mapsto a^{\nu +1}f(az)$ belongs to $\mc{H}(E)$ and has the same norm as $f$.  Such spaces were characterized by L. de Branges in \cite{B2} (see also \cite[Section 50]{Branges}). 

\smallskip

Let $A_{\nu}$ be the even real entire function defined in \eqref{20240710_08:38}. We consider an odd real entire function $B_{\nu}$ given by
\begin{align*}
B_{\nu}(z)  = \Gamma(\nu +1) \left(\tfrac12 z\right)^{-\nu} J_{\nu+1}(z) = \sum_{n=0}^{\infty} \frac{(-1)^n \big(\tfrac12 z\big)^{2n+1}}{n!(\nu +1)(\nu +2)\ldots(\nu+n+1)}.
\end{align*}
Then the function 
\begin{equation*}
E_{\nu}(z) := A_{\nu}(z) - iB_{\nu}(z)
\end{equation*}
turns out to be a Hermite-Biehler function of bounded type in $\C^+$, with no real zeros, and $\mc{H}(E_{\nu})$ is homogeneous of order $\nu$. Moreover we have $\tau(A_{\nu}) = \tau(B_{\nu}) = \tau(E_{\nu}) =1$. Note also that $A_{\nu}, B_{\nu} \notin \mc{H}(E_{\nu})$ (for this one can look at the behaviour of the Bessel function $J_{\nu}$ at infinity). We refer the reader to \cite[Section 50]{Branges}, \cite[Section 5]{HV} or \cite[Sections 3 and 4]{CL1} for details. Some properties of the spaces $\mc{H}(E_{\nu})$ that are relevant for our purposes are collected in the next lemma, which is contained in \cite[Eqs. (5.1), (5.2) and Lemma 16]{HV}. 

\begin{lemma}\label{Sec5_rel_facts}
Let $\nu > -1$. The following properties hold:
\begin{enumerate}
\item[(i)] There exist positive constants $a_\nu,b_\nu$ such that 
\begin{align*}
a_\nu |x|^{2\nu+1} \le |E_{\nu}(x)|^{-2} \le b_\nu |x|^{2\nu+1}
\end{align*}
for all $x \in \R$ with $|x|\geq1$.
\smallskip

\item[(ii)] For $F\in\H(E_\nu)$ we have the identity 
\begin{align*} 
\int_{\R} |F(x)|^{2}\,|E_{\nu}(x)|^{-2}\, \d x = c_\nu \int_{\R} |F(x)|^2 \,|x|^{2\nu+1} \,\d x\,,
\end{align*}
with $c_\nu = \pi \,2^{-2\nu-1}\, \Gamma(\nu+1)^{-2}$.

\smallskip

\item[(iii)] An entire function $F$ belongs to $\H(E_\nu)$ if and only if $F$ has exponential type at most $1$ and
\begin{equation*}
\int_{\R} |F(x)|^2 \,|x|^{2\nu+1}\, \d x <\infty.
\end{equation*}
\end{enumerate}
\end{lemma}
This leads us to the conclusion that $E_{\nu}$ verifies conditions (C1), (C2), (C3) and, with the aid of Lemma \ref{Sec5_rel_facts} and Lemma \ref{nonnegl1-tohe}, one also verifies (C4). We then fall in the framework of Theorem \ref{dB_Thm2}, and the corollary follows (note that one can go from exponential type $1$ to an arbitrary exponential type $\delta >0$ with a dilation argument).

\subsection{Proof of Corollary \ref{Cor6_PWeights}} We proceed with the proof for $\delta=1$. For a generic exponential type $\delta >0$, one simply needs to consider the function $z \mapsto E_{\nu}(\delta z)$ in the proof below.

\smallskip

Let $H$ be the space of entire functions $F$ of exponential type at most $1$ such that 
\begin{align*}
\|F\|_{H}:= \left(\int_{\R} |F(x)|^2\,\d \mu(x) \right)^{1/2} < \infty.
\end{align*}
From Lemma \ref{Sec5_rel_facts} and the hypotheses of Corollary \ref{Cor6_PWeights}, one sees that $H$ and $\mc{H}(E_{\nu})$ are equal as sets. The next piece of information required for this proof is the following result.

\begin{proposition}\label{Prop_17}
Under the hypotheses of Corollary \ref{Cor6_PWeights} we have
\begin{align}\label{20240710_12:46}
\|F\|_{H} \simeq  \|F\|_{\mc{H}(E_{\nu})}.
\end{align}
\end{proposition}
\begin{proof}
Let $K_{\nu}(w,z)$ be the reproducing kernel of the space $\mc{H}(E_{\nu})$. If $F \in H = \mc{H}(E_{\nu})$ one has the following pointwise bound, as a consequence of the reproducing kernel property and the Cauchy-Schwarz inequality,
\begin{align}\label{20240710_12:40}
|F(x)|^2 = \big|\langle F, K_{\nu}(x, \cdot))\rangle_{\mc{H}(E_{\nu})}\big|^2 \leq \|F\|_{\mc{H}(E_{\nu})}^2\, \big\|K_{\nu}(x, \cdot)\big\|_{\mc{H}(E_{\nu})}^2 = \|F\|_{\mc{H}(E_{\nu})}^2\, K_{\nu}(x, x).
\end{align}
Then, from \eqref{20240710_12:40}, the hypotheses of Corollary \ref{Cor6_PWeights}, and Lemma \ref{Sec5_rel_facts},
\begin{align*}
\int_{\R} |F(x)|^2\,\d \mu(x) & = \int_{|x|\leq R} |F(x)|^2\,\d \mu(x)  + \int_{|x| > R} |F(x)|^2\,\d \mu(x) \\
& \lesssim \|F\|_{\mc{H}(E_{\nu})}^2  \int_{|x|\leq R} K_{\nu}(x, x) \,\d \mu(x) +  \int_{|x| > R} |F(x)|^2\,|x|^{2\nu +1}\,\d x\\
& \lesssim \|F\|_{\mc{H}(E_{\nu})}^2.
\end{align*}
This accounts for one of the inequalities needed to prove \eqref{20240710_12:46}.

\smallskip

For the other inequality we argue as follows. If $\nu \geq -1/2$, we have that $F \in H = \mc{H}(E_{\nu})$ implies $F \in L^2(\R)$ and, from the Paley-Wiener theorem, ${\supp}(\widehat{F}) \subset [-\tfrac{1}{2\pi}, \tfrac{1}{2\pi}]$. Then, from the classical Amrein-Berthier-Nazarov uncertainty principle for the Fourier transform \cite{AB, N} (see, alternatively, \cite[Lemma 9]{CGOOORS}), 
\begin{align}\label{20240712:12:04}
\int_{|x|\leq R} |F(x)|^2\,\d x \lesssim \int_{|x|>R} |F(x)|^2\,\d x.
\end{align}
From \eqref{20240712:12:04} we get
\begin{align*}
\int_{\R} |F(x)|^2\,|x|^{2\nu +1}\,\d x & \leq R^{2\nu +1}\int_{|x|\leq R} |F(x)|^2\,\d x + \int_{|x|>R} |F(x)|^2\,|x|^{2\nu +1}\,\d x\\
& \lesssim R^{2\nu +1}\int_{|x|>R} |F(x)|^2\,\d x + \int_{|x|>R} |F(x)|^2\,|x|^{2\nu +1}\,\d x\\
& \lesssim \int_{|x|>R} |F(x)|^2\,|x|^{2\nu +1}\,\d x\\
& \lesssim \int_{\R} |F(x)|^2\,\d \mu(x),
\end{align*}
which is the desired inequality due to Lemma \ref{Sec5_rel_facts} (ii).

\smallskip

Now we consider the case $-1 < \nu < -1/2$. For $0< \varepsilon <1$, using \eqref{20240710_12:40} we get 
\begin{align}\label{20240712_13:32}
\begin{split}
\int_{|x|< \varepsilon} |F(x)|^2\,|x|^{2\nu +1}\,\d x &\leq \|F\|_{\mc{H}(E_{\nu})}^2\int_{|x|< \varepsilon}  K_{\nu}(x, x)\,|x|^{2\nu +1}\,\d x\\
& \leq \left( \max_{|x| \leq 1} K_{\nu}(x, x) \int_{|x|< \varepsilon} |x|^{2\nu +1}\,\d x\right) \|F\|_{\mc{H}(E_{\nu})}^2.
\end{split}
\end{align}
Hence, from Lemma \ref{Sec5_rel_facts} (ii) and \eqref{20240712_13:32}, we can choose $\varepsilon >0$ sufficiently small, independently of $F$, such that 
\begin{align}\label{20240712_15:31}
\begin{split}
\int_{|x|< \varepsilon} |F(x)|^2\,|x|^{2\nu +1}\,\d x \leq \int_{|x|\geq  \varepsilon} |F(x)|^2\,|x|^{2\nu +1}\,\d x. 
\end{split}
\end{align}
Now consider the entire function $E_{\alpha}$, with $\alpha = -\nu -1$, that has exponential type at most $1$. Note from Lemma \ref{Sec5_rel_facts} (i) that 
\begin{align*}
|E_{\alpha}(x)|^2 \simeq |x|^{2\nu+1}
\end{align*}
for $|x| \geq \varepsilon$ (we extended the range from $|x| \geq 1$ to $|x| \geq \varepsilon$ simply by noting that $E_{\alpha}$ has no real zeros). We then note that $F \in H = \mc{H}(E_{\nu})$ implies that $F\,E_{\alpha} \in L^2(\R)$ and has exponential type at most $2$. An application of the Amrein-Berthier-Nazarov uncertainty principle \cite{AB, N} yields
\begin{align}\label{20240712_15:30}
\begin{split}
\int_{\varepsilon \leq |x| \leq R} |F(x)|^2\,|x|^{2\nu +1}\,\d x & \lesssim \int_{\varepsilon \leq |x| \leq R} |F(x)E_{\alpha}(x)|^2\,\d x\\
&  \lesssim \int_{ |x| > R} |F(x)E_{\alpha}(x)|^2\,\d x \lesssim \int_{ |x| > R} |F(x)|^2\,|x|^{2\nu +1}\,\d x.
\end{split}
\end{align}
Then, \eqref{20240712_15:31} and \eqref{20240712_15:30} plainly lead to 
\begin{align*}
\int_{\R} |F(x)|^2\,|x|^{2\nu +1}\,\d x \lesssim \int_{ |x| > R} |F(x)|^2\,|x|^{2\nu +1}\,\d x \lesssim \int_{\R} |F(x)|^2\,\d \mu(x)\,,
\end{align*}
as desired. This concludes the proof.
\end{proof}
From Proposition \ref{Prop_17} we conclude that $H$ is a reproducing kernel Hilbert space and the conditions to invoke Lemma \ref{Lem15_AC} are in place. Letting $K$ be the reproducing kernel of the Hilbert space $H$, one can show (as done in \cite[\S 4.1]{CChiM}), since the measure $\mu$ is even, that $K$ verifies the following properties, for all $w,z \in \C$: (i) $K(w,w) >0$; \, (ii) $K(-w,-z) = K(w,z)$; \, (iii) $K(z,w) = \overline{K(w,z)}$; \, (iv) $\overline{K(\overline{w},\overline{z})} = K(w,z)$. From this it follows that $E$ defined by \eqref{20240709_15:43} - \eqref{20210909_19:00} verifies conditions (C2) and (C3). Since $\log^+|ab| \leq \log^+|a| + \log^+|b|$, we have $\log^+|E(x)| \leq \log^+\big|E(x)|x|^{\nu + \frac12}\big| + \log^+\big||x|^{-\nu - \frac12}\big|$ and an application of Jensen's inequality as in \eqref{20240712_16:00} yields 
$$\int_{\R} \frac{\log^+|E(x)|}{1+x^2} \, \d x < \infty.$$
From Lemma \ref{theorem-krein} we see that $E$ has bounded type in $\C^+$ (condition (C1)), and moreover, since $H = \mc{H}(E_{\nu})$ as sets, with the aid of Lemma \ref{Lem_9_VET} we conclude that $\tau(E) = \tau(K(0,\cdot)) \leq 1$. In fact, we must have $\tau(E) = 1$, since by Lemma \ref{Lem_9_VET} the space $H = \mc{H}(E_{\nu}) = \mc{H}(E)$ admits only functions $F$ with $\tau(F) \leq \tau(E)$, and we know from the start that there exist functions in $F \in \mc{H}(E_{\nu})$ with exponential type equal to $1$. Finally, condition (C4) follows from Lemma \ref{nonnegl1-tohe} applied to $E_{\nu}$ or $E$, depending if we start with the left-hand side or the right-hand side of \eqref{20240712_16:29} being finite, and the fact that $H$ is isometric to $\mc{H}(E)$.

\end{document}